\documentclass[11pt]{amsart}

\usepackage{textcmds}
\usepackage{color}
\usepackage{array}
\usepackage{amsfonts}
\usepackage[leqno]{amsmath}
\usepackage{amsthm}
\usepackage{mathrsfs}
\usepackage[all]{xy}

\usepackage[
	hyperindex,
	pagebackref,
	pdftex,
	pdfdisplaydoctitle,
	pdfpagemode=UseNone,
	breaklinks=true,
	extension=pdf,
	bookmarks=false,
	plainpages=false,
	colorlinks,
	linkcolor=linkblue,
	citecolor=linkblue,
	urlcolor=linkblue,
	pdfmenubar=true,
	pdftoolbar=true,
	pdfpagelabels,
	pdfpagelayout=SinglePage,
	pdfview=Fit,
	pdfstartview=Fit
]{hyperref}

	\definecolor{linkblue}{rgb}{0,0.2,0.6}

\usepackage[backrefs,msc-links,nobysame]{amsrefs}

\makeatletter

\def\bib@div@mark#1{%
 \@mkboth{{#1}}{{#1}}%
	}
\def\print@backrefs#1{%
 \space\SentenceSpace$\leftarrow$\csname br@#1\endcsname
}
\catcode`\'=11 
\renewcommand{\PrintAuthors}[1]{%
 \ifx\previous@primary\current@primary
  \sameauthors\@empty
 \else
  \def\current@bibfield{\bib'author}%
		  \PrintNames{}{}{\scshape #1}%
 \fi
}
\catcode`\'=12
\def\MRhref#1#2{%
 \begingroup
  \parse@MR#1 ()\@empty\@nil%
  \href{\MR@url}{\texttt{\@tempd\vphantom{()}}}%
  \ifx\@tempe\@empty
  \else
   \ \href{\MR@url}{\texttt{(\@tempe)}}%
  \fi
 \endgroup
}%
\def\MR#1{%
 \relax\ifhmode\unskip\spacefactor3000 \space\fi
 \begingroup
  \strip@MRprefix#1\@nil
  \edef\@tempa{\@nx\MRhref{MR\@tempa}{\@tempa}}%
 \@xp\endgroup
 \@tempa
}
\makeatother

\numberwithin{equation}{section}

\newtheorem{theorem}[equation]{Theorem}
\newtheorem{theorem*}{Theorem}
\newtheorem{definition}[equation]{Definition}
\newtheorem{claim}[equation]{Claim}
\newtheorem{lemma}[equation]{Lemma}
\newtheorem{corollary}[equation]{Corollary}

\newtheorem{proposition}[equation]{Proposition}

\newtheorem{question}[equation]{Question}
\newtheorem{example}[equation]{Example}
\newtheorem{remark}[equation]{Remark}
\newtheorem{note}[equation]{Remark}

\newcommand{\C}{\mathbb{C}}

\newcommand{\Z}{\mathbb{Z}}

\newcommand{\Q}{\mathbb{Q}}

\newcommand{\Pro}{\mathbb{P}}

\newcommand{\dblq}{{/\!/}}

\newcommand{\F}{\operatorname{F}}
\newcommand{\Nef}{\operatorname{Nef}}
\newcommand{\NS}{\operatorname{NS}}
\newcommand{\Pic}{\operatorname{Pic}}

\newcommand{\rank}{\operatorname{rank}}
\newcommand{\sL}{\mathfrak{sl}}
\newcommand{\Dl}{\mathbb{D}^{\sL_2}_{\ell,(1,\ldots,1)}}
\newcommand{\SL}{\operatorname{SL}}

\newcommand{\vortwo}{\overline{A}_{g}^{\operatorname{Vor}(2)}}

\newcommand{\mc}[1]{\mathcal{#1}}
\newcommand{\ov}[1]{\overline{#1}}

\begin{document}

\pagenumbering{arabic}

\title{Conformal blocks divisors on $\ov{M}_{0,n}$ from $\sL_2$}

\author{Valery Alexeev}
\email{alexeev@math.uga.edu}

\author{Angela Gibney}
\email{agibney@math.uga.edu}

\author{David Swinarski}
\email{davids@math.uga.edu}


\date{\today}

\begin{abstract}
We study a family of semiample divisors on $\ov{M}_{0,n}$ defined using conformal blocks and 
analyze their associated morphisms.  
\end{abstract}

\maketitle

\section{Introduction}
The coarse moduli space $\ov{M}_{g,n}$ is a projective variety whose points correspond to isomorphism classes of Deligne-Mumford stable $n$-pointed curves of genus $g$.    Our collective intuition for what moduli spaces should look like is based on the spaces $\ov{M}_{g,n}$, and these varieties have been studied from many perspectives.    One of the most important problems is to describe all morphisms admitted by $\ov{M}_{g,n}$.  Maps to projective varieties are given by the section rings of divisors.  A divisor $D$ is \textit{semiample} if some multiple of $D$  is base point free; the map associated to a semiample divisor is therefore a morphism.  Semiample divisors lie in the nef cone, the set of divisors that nonnegatively intersect all curves on the variety.   

For the special case when $g=0$, $\ov{M}_{0,n}$ is conjectured by Hu and Keel to be a ``Mori Dream Space,''
which would imply that every nef divisor on $\ov{M}_{0,n}$ is semiample and that the cone of nef divisors is finitely generated  \cite{HuKeel}.    
The major support for the conjecture comes from our understanding of the
space for $n=4,5,6$.  For example, we know that $\ov{M}_{0,6}$ is a
Mori dream space from two perspectives: Castravet gives an explicit constructive proof
in \cite{Castravet}, and the sweeping theory of 
\cite{BCHM} shows that log Fano varieties are Mori
dream spaces.   Unfortunately 
since $\ov{M}_{0,n}$ is not log Fano
for $n \ge 7$, one cannot directly apply the results of \cite{BCHM}, and extending Castravet's approach to larger values of $n$  
seems daunting.  
   Attempts to prove that the conjectural descriptions
of the nef and effective cones are correct seem mired down in the fiercely
combinatorial nature of these moduli spaces.

Very recently Fakhruddin, in \cite{Fakh}, showed that there is a potentially infinite supply of semiample divisors on $\ov{M}_{0,n}$ that arise as first Chern classes of vector bundles which come from the theory of conformal blocks.
  First defined by
Tsuchiya, Ueno, and Yamada, these bundles are specified by  a simple
Lie algebra $\mathfrak{g}$, a positive integer $\ell$ called the
{\em{level}}, and an appropriately chosen $n$-tuple of dominant
weights $\vec{\lambda}$.  The ranks of these vector bundles are
computed by the famous Verlinde formula.     Fakhruddin has given recursive formulas for the classes of these divisors and for their intersection with certain combinatorially defined curves in called $\F$-curves.     In this work we study some of the simplest examples of these divisors, determine where they lie in the cone of nef divisors, and describe the morphisms associated to them.

We consider a family of divisors on $\ov{M}_{0,n}$ determined by the
Lie algebra $\mathfrak{g}=\sL_{2}$.  In this case, the weights $\vec{\lambda}$ can be identified with an $n$-tuple of nonnegative integers less than or equal to the level $\ell$.  We take all the weights to be equal to one, and examine the family one obtains by varying $\ell$.  As we now explain, this gives a collection of numerical equivalence classes of semiample divisors $\{\Dl: 1 \le
  \ell \le  g\}$  on $\ov{M}_{0,n}$ with $n=2(g+1)$ even.
Indeed, if $n$ is odd, then  by Lemma \ref{oddsumlemma} below, one has that
$\mathbb{D}^{\sL_2}_{\ell, (1,\ldots,1)}$ is trivial.
Also, Fakhruddin has shown that if $2\ell \ge \sum_i \lambda_i$, then
$\mathbb{D}^{\sL_2}_{\ell, \vec{\lambda}}$ is trivial.   
In particular, if $\ell \ge g+1$ then the divisor $\mathbb{D}^{\sL_2}_{\ell, (1,\ldots,1)}$ is trivial.

One can work with these divisors as completely combinatorial objects.  We give a formula for intersecting these divisors with a basis of $1$-cycles on the space
and give a closed form expression for their classes (Theorem
\ref{intersect}, Propositions \ref{reducedformula} and
\ref{rlclosedform}).    Using this information, we study the position
of these divisors in the nef cone and study the morphisms they define.

Although these divisors arise from a simple choice of weights, they have many interesting geometric properties.     The symmetric group $S_n$ acts on $\ov{M}_{0,n}$ by permuting the marked points and the divisors $\Dl$ are $S_n$ invariant.  In particular, by \cite{KM} (and \cite{GibneyCompositio}), each of these divisors is big, and hence the morphism it defines is birational.   Divisors that give embeddings lie on the interior of the nef cone.  As we show, each member $\Dl$ of our family lies on the boundary of $\Nef(\ov{M}_{0,n}/S_n)$, and defines a birational contraction of $\ov{M}_{0,n}/S_n$.   In particular, four of the divisors generate extremal rays of the symmetric nef cone (Theorem \ref{extremal}).   

These divisors define geometrically meaningful
maps.  For example:
\begin{itemize}
\item
$\mathbb{D}^{\sL_2}_{1,(1,\ldots,1)}$ defines a morphism from $\ov{M}_{0,n}$ to the Satake compactification of the moduli space of abelian varieties of dimension $g$ (Theorem \ref{Satake}). 
\item  $\mathbb{D}^{\sL_2}_{g-1,(1,\ldots,1)}$ defines the morphism from $\ov{M}_{0,n}$ to a natural (generalized) flip of the GIT quotient $(\mathbb{P}^1)^n\dblq \SL_2$ (Theorem \ref{Dg-1morphism}).
\item  $\mathbb{D}^{\sL_2}_{g,(1,\ldots,1)}$ defines the morphism from $\ov{M}_{0,n}$ to the GIT quotient $(\mathbb{P}^1)^n\dblq \SL_2$ (Theorem \ref{GIT}) .
\end{itemize}  

In addition, we recall that Fakhruddin shows that the morphism given by any $\Dl$ for
$\ell \leq g-1$ factors through Hassett's reduction morphisms (\cite{HassettWeighted})
$\rho_{A}: \ov{M}_{0,n} \rightarrow \ov{M}_{0,A}$ for $A=
(\frac{1}{\ell+1},\ldots,\frac{1}{\ell+1})$ \cite{Fakh}*{Prop. 4.7}. 

Using each of the four divisors $\Dl$, $\ell = 1,2,g-1,g$, we obtain a family of nef
divisors on $\ov{M}_{2(g+1)}$, the moduli space of stable curves of
genus $2(g+1)$; we give these classes explicitly in Proposition \ref{flagpullback}.   

 The Ray Theorem of Keel and McKernan \cite{KM}*{Thm. 1.2} and its extension by Farkas and Gibney \cite{FarkasGibney}*{Thm. 4} 
 enable one to give a combinatorial proof that divisors in what we call the log canonical part of the cone are nef.  If a divisor $D$ is of this form, and if $D$ intersects every
$\operatorname{F}$-curve nonnegatively, then $D$ is nef.  As we show in Proposition \ref{nlc},  there are divisors in our family that do not lie in this part of the cone.  In particular, we know no combinatorial proof that they are semiample.


%

There are two other families of conformal blocks divisors that have
been studied extensively.  In \cite{Fakh}, Fakhruddin considers the
set of divisors given by $\sL_2$, level $\ell=1$ and varying weights
$\vec{\lambda}$.  He shows these divisors form a basis for
$\Pic(\ov{M}_{0,n})$ and that all of the morphisms factor through the
reduction maps from $\ov{M}_{0,n}$ to Hassett's moduli spaces of
stable weighted pointed rational curves.  In \cite{agss}, the set of
divisors $\mathbb{D}^{\sL_n}_{1,j}$ given by $\sL_n$, level $\ell=1$
and symmetric sets of fundamental dominant weights
$\vec{\lambda}=\{\omega_j, \ldots, \omega_j\}$ is considered.  The
authors show that each of these divisors is extremal in the symmetric
nef cone and so defines a morphism to a variety of Picard number one.
In \cite{Giansiracusa}, Giansiracusa describes the images of the morphisms given by the $\mathbb{D}^{\sL_n}_{1,j}$.

{\textit{Outline of paper:}} 
In Section \ref{M} we give definitions and references for divisors and curves on $\ov{M}_{g,n}$.  In Theorem \ref{3curvefamilies} we define three families of independent curves on $\ov{M}_{0,n}$ that we use to  prove that the divisors $\Dl$ are extremal 
in the symmetric nef cone.  In Section  \ref{general} 
we provide references and brief definitions for general conformal blocks.  Some of our main technical lemmas about
$\sL_2$ conformal blocks are proved in Section \ref{sl2bundles}, where we recall the Verlinde formula, as well as the factorization and fusion rules.  In Section \ref{intersectwithbasis} we give the intersection formulas for the $S_n$-invariant divisors $\Dl$
with a basis for the $1$-cycles on $\ov{M}_{0,n}/S_n$.  As an application we show that the set of nontrivial divisors $\Dl$ forms a basis of $\Pic(\ov{M}_{0,n}/S_n)$.   In Section \ref{classes} we prove Proposition \ref{reducedformula}, which
gives a general formula for the classes of the divisors 
$\mathbb{D}^{\sL_2}_{\ell,(1,\ldots,1)}$.  This expression depends on 
ranks of certain $\sL_2$ conformal blocks bundles which we give in  Proposition \ref{rlclosedform}.  In Section \ref{extremalformulae} we give simplified expressions  for the four elements of the family that
generate extremal rays of the symmetric nef cone and we give simplified versions of the classes of some of the divisors
that lie on higher dimensional extremal faces.  In Section \ref{position} we show that the divisors $\Dl$ are extremal in 
the symmetric nef cone and that they reside in the part of the cone
that was not previously well understood.  
Finally, in Section \ref{morphisms} we study morphisms defined by the extremal divisors.

{\bf Acknowledgements: } We would like to thank Boris Alexeev for
writing programs which were used to compute examples of $\sL_2$ conformal
blocks bundles.  

\section{Divisors and curves on the moduli stack}\label{M}
The stack $\ov{\mc{M}}_{g,n}$, which parametrizes flat families of stable curves, reflects the geometry of the variety $\ov{M}_{g,n}$; throughout this paper, we will sometimes work on  the space and sometimes on the stack.  In genus $0$, $\ov{M}_{0,n}$ is actually a fine moduli space, and so these two points of view are equivalent in this particular case.

The conformal blocks divisors are defined on the moduli stack $\ov{\mc{M}}_{g,n}$ as the
determinants of vector bundles of conformal blocks.  We will often express them in terms of the tautological and boundary classes.  In Section \ref{tautological classes}, we record the basic definitions of these classes and refer the reader to \cite{ArbarelloCornalbaCombinatorial} for details.  In Section \ref{independentcurves} we define three collections of curves on $\ov{M}_{0,n}$ that we prove are independent.  These curves are used to analyze the divisors $\Dl$.

\subsection{Divisor classes}\label{tautological classes}
We write $\lambda$ for the first Chern class
of the Hodge bundle, as is standard in the literature.  
  For  $1 \le i \le n$, we denote  by $\sigma_i$ the $n$ sections of the universal family
$\pi: \ov{\mc{M}}_{g,n+1} \longrightarrow \ov{\mc{M}}_{g,n} $.  Then for $\omega_{\pi}$  the rank $1$ relative
dualizing sheaf, one has the tautological classes
$\psi_i=c_1(\sigma_i^*(\omega_{\pi}))$.  We refer to the sum $\Psi=\sum_{i=1}^n\psi_i$ as the total $\Psi$ class.  The divisor $\kappa=\kappa_1=\pi_*(\omega_{\pi}^2)$ is ample.

We write $\delta_0$ for the class of the boundary component $\Delta_0$, the divisor whose general element has a single nonseparating node.  For $0 \le i \le \lfloor \frac{g}{2}\rfloor$, and $J \subset \{1,\ldots,n\}$, let $\delta_{i, J}$ be the class of the boundary divisor $\Delta_{i, J}$.   The general element of $\Delta_{i, J}$ has a single separating node that breaks the curve into two components, one of which is a curve of genus $i$ and has $|J|+1$ marked  points  consisting of an attaching point together with points labeled by the set $J$.
 If $n=0$, it is customary to write $\delta_i$ instead of $\delta_{i,\emptyset}$ and if $g=0$, it is customary to write
$\delta_J$ rather than $\delta_{0,J}$.  By $\Delta$ we mean the sum of all the boundary divisors.

As we often work with $S_n$-invariant divisor classes on $\ov{M}_{0,n}$, we often find it convenient to write
$$B_j=\sum_{\stackrel{J \subset \{1,\ldots,n\}}{|J|=j}}\delta_J.$$
In \cite{KM}*{Thm. 1.3}, the $B_j$ were shown to generate the extremal rays of the cone of effective divisors of the quotient $\ov{M}_{0,n}/S_n$.  These form
a basis for $\Pic(\ov{M}_{0,n})^{S_n}$.

Finally, by \cite{Rahul}*{Theorem 1} the class of the canonical divisor in this basis is 
$$K_{\ov{M}_{0,n}}=\Psi-2\Delta=\sum_{i=1}^{\lfloor \frac{n}{2} \rfloor}
(\frac{i(n-i)}{(n-1)}-2)B_i.$$

\subsection{Three collections of independent curves}\label{independentcurves}

An $\F$-curve on $\ov{M}_{0,n}$ is any curve  that is numerically equivalent to a $1$-dimensional component of the boundary.  Its class depends only on a partition of the marked points into four nonempty sets.    
  As we consider the intersection of $\F$-curves with
  symmetric divisor classes on $\ov{M}_{0,n}$,  we do not need to know the
  partition itself, but only the size of the cells of the partition.
  Thus, a partition $a+b+c+d=n$ of the integer $n$ into four
  positive integers determines an F-curve class, up to $S_{n}$
  symmetry.  We denote such a curve by $F_{a,b,c,d}$, or even more
  briefly by $F_{a,b,c}$, with the fourth integer $d=n-(a+b+c)$ implicit.
Details and a definition of 
  $\F$-curve on $\ov{\mc{M}}_{g,n}$ for all $g$ and all $n$ are given in \cite{GKM}*{Thm. 2.2, Fig. 2.3}.

We next define three families of $\F$-curves on $\ov{M}_{0,n}$ that we show are independent.  These curves are used to demonstrate in Theorem \ref{extremaldivisors} that the divisors $\Dl$ are extremal 
in the symmetric nef cone.

 \begin{theorem}\label{3curvefamilies} Let $n=2g+2$, or $n=2g+3$.  Each of the following three sets consist of independent curves.
 \begin{enumerate}
 \item\label{C1} $\mathscr{C}_1=\{F_{1,1,i}: 1 \le i \le g\}$;
 \item\label{2curves} $\mathscr{C}_2=\{F_{2,2,i}: 1 \le i \le g-1\}$; and
 \item\label{oddcurves} $\mathscr{C}_3=\{F_{3,3,2i+1}: 0 \le i \le k-2\} \cup \{F_{1,1,2i+1}: 0 \le i \le k-1\}$, and $g=2k$ or $g=2k-1$.
 \end{enumerate}
 \end{theorem}

This leads to the following well-known corollary:

\begin{corollary}\label{abasis} 
Let $n=2(g+1)$ or $n=2(g+1)+1$ and $\mathscr{C}_1$ as in Theorem \ref{3curvefamilies}. 
Then $\mathscr{C}_1$ is a basis for $N_1(\ov{M}_{0,n}/S_n, \mathbb{Q})$.
\end{corollary}

\begin{proof}[Proof of Corollary \ref{abasis}]  The N\'eron-Severi space of $\operatorname{NS}(\ov{M}_{0,n}/S_n)_{\Q}$ is $g$-dimensional.  By Theorem \ref{3curvefamilies}, $\mathscr{C}_1$ consists of $g$ independent curves.
\end{proof}
 
\begin{proof}[Proof of Theorem \ref{3curvefamilies}]  For a proof that the first family $\mathscr{C}_1$ is independent, see \cite{agss}*{Prop. 4.1}.

To show that $\mathscr{C}_2$ is a collection of independent curves, suppose that $R=\sum_{i=1}^{g}b_i F_{2,2,i}$ is equivalent to zero and put $B=\sum_{i=1}^{g}b_i$.  By intersecting $R$ with the boundary classes $B_i$, one gets the following equations
\begin{enumerate}
\item $0= R \cdot B_2 = -2B -b_2$;
\item $0= R \cdot B_3 = 2 b_{1} -b_3$;
\item $0=R \cdot B_4 = B + 2b_2-b_4$;
\item for  $5 \le i \le g-3$, and if $n=2g+3$ and $i\in \{g-2,g-1\}$, $0= R \cdot B_i = 2 b_{i-2} -b_i-b_{i-4}$;
\item for $n=2g+2$,
\begin{enumerate}
\item $0= R \cdot B_{g-2} = 2 b_{g-4} -b_g-b_{g-2}-b_{g-6}$;
\item $0= R \cdot B_{g-1} = 2 b_{g-3} -2b_{g-1}-b_{g-5}$;
\item $0=R \cdot B_g= b_{g-2} - b_g-b_{g-4}$;
\item $0=R \cdot B_{g+1}= 2b_{g-1} - b_{g-3}$;
\end{enumerate}
\item for $n=2g+3$,
\begin{enumerate}
\item $0=R \cdot B_g= 2b_{g-2} - b_g-b_{g-1}-b_{g-4}$;
\item $0=R \cdot B_{g+1}= 2b_{g-1} - b_{g-3}-b_{g-2}$.
\end{enumerate}
\end{enumerate}
By adding all of these relations together one gets that $B=-b_1$.  We now show that all of the coefficients are a multiple of $b_1$.
Iteration of equation $(4)$ gives that for  $5 \le i \le g-3$, and if $n=2g+3$ and $i\in \{g-2,g-1\}$, then
for $m \le \lfloor \frac{i+1}{2} \rfloor$, one has $b_i=mb_{i-2m+2}-(m-1)b_{i-2m}$.   For $i$ odd, taking $m=\frac{i-1}{2}$, and using equation $(2)$ gives that $b_i=\frac{i+1}{2}b_1$. For $i$ even, taking $m=\frac{i-2}{2}$, and using equations $(3)$ and then $(1)$ gives that 
$$b_i=\frac{i-2}{2}b_4-\frac{i-6}{2} b_2=\frac{i-2}{2}(B+2b_2)-\frac{i-6}{2} b_2 \\=\frac{i-2}{2}B+\frac{i+2}{2} b_2=-\frac{i+6}{2}B=\frac{i+6}{2}b_1.$$
Using the remaining equations it is very easy to get expressions for the rest of the coefficients in terms of $b_1$.  To finish the proof, we intersect $R$ with  the divisor $\psi=\sum_{i=1}^n\psi_i$.  Since $R$ is assumed trivial, $R\cdot \psi=0$.  On the other hand, one can easily check that $\psi$ intersects $R$ in degree $b_1$, proving that $b_1$ must be zero.  We conclude that all the coefficients are zero, and the curves are independent.

To show that $\mathscr{C}_3$ consists of $2k-1$ independent curves,
suppose that $$R=\sum_{i=0}^{k-1}a_i F_{1,1,2i+1} + \sum_{i=0}^{k-2}b_i F_{3,3,2i+1}$$ is equivalent to zero.  Put $A=\sum_{i=0}^{k-1}a_i$ and  $B=\sum_{i=0}^{k-2}b_i$.  By intersecting $R$ with the boundary classes $B_i$, one gets the following equations
\begin{enumerate}
\item $0= R \cdot B_2 = A+2a_0$;
\item $0= R \cdot B_3 = -2B - a_0 -(a_1+b_1)$;
\item $0= R \cdot B_{2j} = 2(a_{j-1}+b_{j-2})$,  for  $2 \le j \le k$ and $j \ne 3$.
\item $0= R \cdot B_{2j+1} = -(a_{j-1}+b_{j-3}+a_j+b_j)$,  for  $2 \le j \le k$;
\item $0=R \cdot B_6 = B + 2(a_2+b_1)$.
\end{enumerate}
By adding all of these relations together one gets that $A+a_0+b_0-B=0$.   Intersecting $R$ with $\psi=\sum_{i=1}^n\psi_i$ gives that $0=A+a_0+b_0$.  Putting these two facts together gives that $B=0$ and using the first equation, which says that $A=-2a_0$, we can write $b_0=a_0$.  We will next show $b_j=0$ for $1 \le j \le k-1$.  We work by induction on $j$, with base case $j=1$.  To begin with, we can see that $a_0=-a_1$. Indeed, we have seen that $b_0=a_0$ and equation $(3)$  says that $b_0=-a_1$.  Now equation $(2)$ says that $b_1=-a_0-a_1$, and so the base case holds.    So we fix $j \in \{2, \ldots, k-1\}$ and  assume that $b_{\ell}=0$ if $\ell < j$.  Now by equations $(4)$,  and then $(3)$,
one has that $b_j=-(a_{j-1}+b_{j-3}+a_j)=b_{j-2}+b_{j-3}+b_{j-1}$, which is zero by our induction hypothesis.
Finally, since $b_j=0$, we may use equation $(3)$ to see that the $a_j$ are zero as well.  
\end{proof}

\section{Conformal blocks}\label{CBbackground}
The facts we use about conformal blocks divisors can be found in \cite{Beauville}, \cite{LooijengaNotes}, \cite{Ueno}, and \cite{Fakh}.  In 
Section \ref{general} we give an informal description of conformal blocks.  In Section \ref{sl2bundles} we give specific results about $\sL_2$ conformal blocks,  including the factorization and fusion rules and other technical lemmas
we use throughout the paper.
\subsection{General definition}\label{general}
Let $\mathfrak{g}$ be a simple Lie algebra, $\ell$ a positive integer called the {\it{level}}, and $\vec{\lambda}=(\lambda_1,\ldots, \lambda_n)$ an $n$-tuple of dominant integral weights for $\mathfrak{g}$ of level $\ell$.   Conformal blocks divisors are determinants of vector bundles $\mathbb{V}(\mathfrak{g},\ell, \vec{\lambda})$ defined on the stack $\ov{\mc{M}}_{g,n}$ which are determined by the data $\mathfrak{g}$, $\ell$, and $\vec{\lambda}$.  We write:
$$\mathbb{D}^{\mathfrak{g}}_{\ell, \vec{\lambda}}=c_1(\mathbb{V}(\mathfrak{g},\ell, \vec{\lambda})).$$  
We often refer to the $\mathbb{D}^{\mathfrak{g}}_{\ell,
  \vec{\lambda}}$ as CB-divisors and the
$\mathbb{V}(\mathfrak{g},\ell, \vec{\lambda})$ as CB-bundles.  The
divisors $\mathbb{D}^{\mathfrak{g}}_{\ell, \vec{\lambda}}$
are also known in the literature as generalized theta divisors.

We give two descriptions of the fiber of the CB-bundle $\mathbb{V}(\mathfrak{g},\ell, \vec{\lambda})$ over a smooth point $(C, \overline{p})=(C, p_{1},\ldots, p_n) \in \mathcal{M}_{g,n}$.
First, a geometric description \cite{LaszloSorger}*{Eq. (1.2.2)}:
$$\mathbb{V}(\mathfrak{g},\ell, \vec{\lambda})|_{(C, \overline{p})} \cong H^0(\mathcal{M}_{\mathfrak{g}}^{par}(C,\overline{p}), \mathcal{L}).$$
\noindent
Here $\mathcal{M}_{\mathfrak{g}}^{par}(C,\overline{p})$ is a moduli stack parametrizing quasi-parabolic $\mathfrak{g}$-bundles on $C$
determined by $\vec{\lambda}=(\lambda_1,\ldots,\lambda_n)$.  The line bundle $\mathcal{L}$ is determined by $\ell$.

Second, we give an algebraic description, which comes from the construction of these CB-bundles using the representation theory of affine Lie algebras.
 To define them, let 
\[\widehat{\mathfrak{g}} = (\ \mathfrak{g} \otimes \C((z_{i})) )\oplus \C c
\]
be the affine Lie algebra associated to $\mathfrak{g}$.  As for finite-dimensional Lie algebras, to each weight $\lambda_{i}$ there is an associated irreducible  $\widehat{\mathfrak{g}}$-module $\mathcal{H}_{\lambda_{i}}$.  Let $\mathcal{H}_{\vec{\lambda}}=\mathcal{H}_{\lambda_{1}} \otimes
 \cdots \otimes \mathcal{H}_{\lambda_{n}}$.
For a smooth point $(C; \{p_1,\ldots, p_n\}) \in \mathcal{M}_{g,n}$, set  $U = C-\{p_{1}, \ldots, p_{n} \}$.  Choose a local coordinate at each point $p_{i}$.  This gives rise to a ring homomorphism $\mathcal{O}(U) \rightarrow \C((z))$ for each $i=1,\ldots,n$ mapping a function $f$ to its Laurent series expansion $f_{p_i}$ around $p_i$.  
  A $\mathfrak{g} \otimes \mathcal{O}(U)$ action on $\mathcal{H}_{\vec{\lambda}}$ is defined by the formula 
\[   (X \otimes f) \cdot (v_{1} \otimes \cdots \otimes v_{n}) = \sum_{i=1}^{n}  v_{1} \otimes \cdots \otimes (X \otimes f_{p_{i}}) v_{i} \otimes \cdots \otimes v_{n},
\]
and the fiber of $\mathbb{V}(\mathfrak{g},\ell, \vec{\lambda})$ over the point $(C, p_{1},\ldots, p_n)$ is the vector space of coinvariants $\mathcal{H}_{\vec{\lambda}}/(\mathfrak{g} \otimes \mathcal{O}(U))\mathcal{H}_{\vec{\lambda}}$.    These fibers are independent of the choice of local coordinate around $p_i$ and vary nicely in families to form an algebraic vector bundle on $\mathcal{M}_{g,n}$.  The construction can be extended to nodal curves, and yields an algebraic vector bundle $\mathbb{V}(\mathfrak{g},\ell, \vec{\lambda})$ on the stack $\ov{\mc{M}}_{g,n}$.

\subsection{$\sL_2$ CB-bundles}\label{sl2bundles}

In this section we recall the Verlinde formula, as well as the factorization and fusion rules
for $\sL_2$ CB-bundles.  Using factorization and the fusion rules, we
prove  Lemmas \ref{oddsumlemma},
\ref{GeneralizedTriangleInequality}, and \ref{mu rank lemma}, which are all results about ranks  of particular $\sL_2$ CB-bundles that considered in this work.  We also prove Lemma \ref{mu lemma}, which gives a simple formula for the degrees of a certain CB bundles on 
$\mathbb{P}^1 \cong \ov{M}_{0,4}$.   

The root system of $\mathfrak{g}=\sL_2$ may be identified with $\Z$,
and  dominant integral weights $\lambda_i$ of
level $\ell$ are simply nonnegative  integers $0 \le \lambda_i \le
\ell$.  Let $ \vec{\lambda}=(\lambda_1,\ldots,\lambda_n)$ be a vector
of dominant integral weights of level $\ell$.  

Throughout this section work only with $\sL_2$, and we work with an
arbitrary (but fixed) level $\ell$.  Therefore we will abbreviate our
notation a little.

\textbf{Notation}  We write 
\begin{equation}
 r_{\vec{\lambda}} := \rank \mathbb{V}(\sL_2,\ell, \vec{\lambda}).
\end{equation}

For $ \vec{\lambda}=(\lambda_1,\ldots,\lambda_n)$, the rank of the vector bundle $\mathbb{V}(\sL_2,\ell, \vec{\lambda})$ is given by the Verlinde formula:
\begin{equation} \label{Verlinde formula}
r_{\vec{\lambda}} =(\frac{\ell+2}{2})^{g-1} \sum_{j=0}^{\ell} 
\frac{\prod_{i=1}^n
  \sin(\frac{(\lambda_i+1)(\lambda_j+1)\pi}{\ell+2})}
{(\sin(\frac{(\lambda_j+1)\pi}{\ell+2}))^{2g+n-2}} .
\end{equation}
\noindent
Although this formula is quite elegant, it is often computationally
more efficient to use the factorization rules.  These may be stated
for any simple Lie algebra $\mathfrak{g}$, but we will only work with $\mathfrak{g}=\sL_2$.

\begin{proposition}[Propagation for CB-bundles]
Let $\vec{\lambda} = (\lambda_1,\ldots,\lambda_n)$, and suppose that
$\lambda_n=0$.  Then $\mathbb{V}(\mathfrak{g},\ell,\vec{\lambda}) =
\pi_n^{*} \mathbb{V}(\mathfrak{g},\ell,\hat{\lambda})$, where
$\hat{\lambda} = (\lambda_1,\ldots,\lambda_{n-1})$ and $\pi_n:
\ov{M}_{g,n} \rightarrow \ov{M}_{g,n-1}$ is the map forgetting the $n$th
marked point.  In particular, $r_{\vec{\lambda}} = r_{\hat{\lambda}}$.
\end{proposition}

\begin{proposition}[Factorization for $\sL_2$ CB-bundles] Let
  $\vec{\mu} \cup \vec{\nu}$ be a partition of the vector
  $\vec{\lambda}=(\lambda_1,\ldots,\lambda_n)$ into two vectors each
  of length at least 2.  Then  
$$ r_{\vec{\lambda}}=\sum_{\alpha =0}^{\ell}  r_{\vec{\mu} \cup\alpha}  r_{\vec{\nu} \cup \alpha}.$$
\end{proposition}
\noindent
(In general, statements of factorization for CB-bundles insert weights $\alpha$ and $\alpha^{*}$, where $*$ denotes the involution on the root lattice defined by $\alpha^{*} = \sigma(\alpha)$, where $\sigma \in W(\mathfrak{g})$ is the longest word in the Weyl group of $\mathfrak{g}$.  For $\mathfrak{g} = \sL_2$, one has $\alpha^*=\alpha$.)

Use of the factorization rules depends 
on the one, two, and three point fusion rules, that is, the ranks of conformal blocks bundles for $n=1$,$2$, and $3$. These are well-known for $n=1$ and $2$ for arbitrary $\mathfrak{g}$ and for $n=3$ when $\mathfrak{g}=\sL_2$; see for instance \cite{Beauville}*{Lemma 4.2, Cor. 4.4}.  We state them for $\sL_2$ below.

\begin{proposition}[Fusion rules for $\sL_2$.] \label{sl2 fusion
    rules}  Write
  $r_{\vec{\lambda}}= \operatorname{rk}(\mathbb{V}(\sL_2,\ell,
  \vec{\lambda}))$.

If $n=1$:
 \begin{displaymath}r_{(\lambda)}= \left\{
\begin{array}{l}
1 \quad \mbox{ if $\lambda= 0$}\\
0 \quad \mbox{ otherwise.}\\
\end{array} \right.
\end{displaymath}
If $n=2$:
 \begin{displaymath}  r_{(\lambda_1,\lambda_2)} = \left\{
\begin{array}{l}
1 \quad \mbox{ if $\lambda_1= \lambda_2$}\\
0 \quad \mbox{ otherwise.}\\
\end{array}\right.
\end{displaymath}
If $n=3$:
\begin{displaymath}r_{(\lambda_1,\lambda_2, \lambda_3)}= \left\{
\begin{array}{l}
1 \quad \mbox{ if $\sum_{i=1}^3\lambda_i \equiv 0 \bmod 2$, $\sum_{i=1}^3\lambda_i \leq 2\ell$, and $\lambda_{i} \leq \frac{1}{2}\sum_{i=1}^3\lambda_i$.}\\
0 \quad \mbox{ otherwise.}\\
\end{array}\right.
\end{displaymath}

\end{proposition}

We now present four small technical results that are used throughout the paper.

\begin{lemma}[Odd Sum Rule] \label{oddsumlemma} Let $\mathfrak{g} = \sL_2$, let $\ell$ be an arbitrary level, and let $\vec{\lambda} = (\lambda_{1},\ldots,\lambda_{n}).$  If $\sum_{i=1}^{n} \lambda_i$ is odd, then $ r_{\vec{\lambda}}=0.$
\end{lemma}

\begin{proof} 
We argue by induction on $n$.  The cases $n=1$, $2$, and $3$ follow
from the fusion rules for $\sL_2$ (see Propostion \ref{sl2 fusion
  rules} above).  So suppose $n \geq 4$ and that the statement is true for $n-1$.  We can find two weights $a = \lambda_{i}$ and $b = \lambda_{j}$ in $\vec{\lambda}$ such that $a+b$ is even.  
Since the rank $r_{\vec{\lambda}}$ is symmetric with respect to permutations of the
$\lambda_i$, we may assume without loss of generality that $\{i,j\} =
\{1,2\}$.  Let $\vec{\mu} = (\lambda_3,\ldots,\lambda_n)$ be
complementary vector.  Apply the factorization formula with the
partition $\vec{\lambda} = (a,b) \cup \vec{\mu} $:
\begin{equation} \label{factorodd}   
r_{\vec{\lambda}}
=  \sum_{\alpha=0}^{\ell} r_{(a,b, \alpha) }  r_{\vec{\mu} \cup \alpha} .
\end{equation} 
Since $a+b$ is even, whenever $\alpha$ is odd, we have 
$r_{(a,b,\alpha)} =0$ by the three point fusion rules.  We also know
that $\sum_{i=3}^{n} \lambda_{i}$ is odd, since $\sum_{i=1}^{n}
\lambda_i$ is odd and $a+b$ is even.  But then, by induction,
$r_{\vec{\mu} \cup \alpha}  = 0$ whenever $\alpha$ is even.  Therefore
all the summands in (\ref{factorodd}) are zero, and $r_{\vec{\lambda}}
= \operatorname{rk}(\mathbb{V}(\sL_2,\ell, \vec{\lambda}))=0$, as claimed.
\end{proof}

\begin{lemma}[Generalized Triangle
  Inequality]\label{GeneralizedTriangleInequality}  Let $\mathfrak{g}
  = \sL_2$, let $\ell$ be an arbitrary level, and let $\vec{\lambda}=
  (\lambda_{1},\ldots,\lambda_{n}).$   If there exists $i\in
  \{1,\ldots,n\}$ such that $\lambda_{i} > \sum_{j \neq i} \lambda_j$,
  then $r_{\vec{\lambda}} = \operatorname{rk}(\mathbb{V}(\sL_2,\ell, \vec{\lambda})) = 0$.
\end{lemma}
\begin{proof}
This can be proved using factorization and induction, with $n=3$ as the base case.
\end{proof}
\begin{lemma}\label{mu rank lemma}
Let $\vec{\mu} = (\mu_{1},\mu_{2},1,1)$ be a vector of weights for
$\sL_2$ satisfying $0 \leq \mu_i \leq \ell$ for $i=1,2$.  Then
\begin{equation}
\operatorname{rk}(\mathbb{V}(\sL_2,\ell, \vec{\mu})) = \left\{\begin{array}{l}
2 \quad \mbox{ if $\mu_1 = \mu_2$ and $\mu_1 \not\in \{0,\ell\}$} \\
1 \quad \mbox{ if $\mu_1 = \mu_2$ and $\mu_1 \in \{0,\ell\}$} \\
1 \quad \mbox{ if $\mu_2 = \mu_1 \pm 2$}. \\
0 \quad \mbox{ otherwise}
\end{array} \right.
\end{equation}
\end{lemma}
\begin{proof}
First, we consider $\operatorname{rk}(\mathbb{V}(\sL_2,\ell,
\vec{\mu}))$.  By factorization applied to the partition $\vec{\mu} = (\mu_1, \mu_2) \cup (1, 1)$ we obtain
\begin{equation} \label{factored} r_{ \vec{\mu}}= \sum_{\alpha=0 }^{\ell} 
r_{ (\mu_1, \mu_2, \alpha)} r_{(1,1, \alpha )}.
\end{equation}
By the two and three point fusion rules for $\sL_2$, we have:
\begin{eqnarray*}
r_{(1,1,0)} & = & 1 \\ 
r_{(1,1,1)} & = & 0 \\
r_{(1,1,2)} & = & 1 \\
r_{(1,1,\alpha)} & = & 0, \mbox{ if $\alpha>2$.} 
\end{eqnarray*}
\noindent Thus only two summands on the right hand side of (\ref{factored}) are possibly nonzero.  We have:
\begin{displaymath}
r_{\vec{\mu}} = r_{(\mu_1,\mu_2,0)} r_{(1,1,0)} + r_{(\mu_1,\mu_2,2)}r_{(1,1,2)} = r_{(\mu_1,\mu_2)} + r_{(\mu_1,\mu_2,2)}.
\end{displaymath}
By the two point fusion rules for $\sL_{2}$, $r_{(\mu_1,\mu_2)}=0$
unless $\mu_1 = \mu_2$, and then this rank is $1$.  By the three point
fusion rules, we have $r_{(\mu_1,\mu_2,2)}=0$ unless $\mu_2 \in
\{\mu_{1}-2, \mu_{1},\mu_{1} +2\}$.  Also, $r_{(0,0,2)}=0$, and
$r_{(\ell,\ell,2)} = \operatorname{rk}(\mathbb{V}(\sL_2,\ell, (\ell,\ell,2)))=0$.  The result follows.
\end{proof} 

\begin{lemma}\label{mu lemma}
Let $\vec{\mu} = (\mu_{1},\mu_{2},1,1)$ such that for $i\in \{1,2\}$, $\mu_i \in \mathbb{Z}$ and $0 \le \mu_i \le \ell$.   Consider the CB vector bundle $\mathbb{V}(\sL_2, \ell, \vec{\mu})$ on $\ov{M}_{0,4} \cong \Pro^{1}$.  Then 
\begin{equation}
\operatorname{deg} \mathbb{V}(\sL_2,\ell, \vec{\mu})  = \left\{\begin{array}{l}
0 \quad \mbox{ if   $\vec{\mu} \ne  (\ell,\ell,1,1)$} \\
1 \quad \mbox{ if   $\vec{\mu} =  (\ell,\ell,1,1)$ }.
\end{array} \right.
\end{equation}
\end{lemma}

\begin{proof}
We use \cite{Fakh}*{Cor.3.4, Formula (3.9)}, with the notation $\operatorname{rk}(\mathbb{V}(\sL_2,\ell, \vec{\lambda}))=r_{\lambda}$.
\begin{eqnarray} \label{fakh deg formula}
\deg \mathbb{V}(\sL_2,\ell, \vec{\mu}) & =& \frac{1}{2(\ell+h^{\vee})} \left\{ \left\{ r_{\vec{\mu}} \sum_{i=1}^4 c(\mu_{i}) \right\} \right.-  \nonumber \\
&& \left. \left\{ \sum_{\alpha=0}^{\ell} c(\alpha) \left\{ r_{(\mu_1,\mu_2,\alpha)}  r_{(\mu_3,\mu_4,\alpha)} +   r_{(\mu_1,\mu_3,\alpha)}  r_{(\mu_2,\mu_4,\alpha)} + r_{(\mu_1,\mu_4,\alpha)}  r_{(\mu_2,\mu_3,\alpha)} \right\} \right\} \right\}.
\end{eqnarray}
Here $c(\alpha)$ denotes the Casimir scalar associated to $\alpha$,
and $h^{\vee}$ denotes the dual Coxeter number of $\sL_2$.  
For $\sL_2$, these are given by $c(\alpha) = \alpha^2/2 + \alpha$, and $h^{\vee}=2$.  

We only want to compute $\deg \mathbb{V}(\sL_2,\ell, \vec{\mu})$ in
cases where $r_{\vec{\mu}} = \rank \mathbb{V}(\sL_2,\ell, \vec{\mu})
>0$.  These cases are listed in the previous lemma.  Let's check for
instance that if $\mu_{1}=\mu_{2} = \mu, \mu \not\in \{0,\ell\}$,  then $\mathbb{V}(\sL_2,\ell, \vec{\mu}) = 0$:

Recall that the three point ranks for $\sL_2$ are always 0 or 1.  Note that $r_{(\mu_3,\mu_4,\alpha)}=0$ unless $\alpha=0,2$.  However, $c(0)=0$, so the first term of the second line of Fakhruddin's formula contributes only when $\alpha=2$.  By the symmetry of the vector of weights $\mu$, the second two terms of the second line are the same, and $r_{(\mu,1,\alpha)}=0$ unless $\alpha=\mu \pm 1$.  Then we get
\begin{eqnarray*}
\deg\mathbb{V}(\sL_2,\ell, \vec{\mu}) & = & \frac{1}{2(\ell+h^{\vee})} \left\{ \left\{  r_{\vec{\mu}} \sum_{i=1}^4 c(\mu_{i}) - c(2) - 2 c(\mu-1)  -2 c(\mu+1) \right\} \right\} \\
& = & \frac{1}{2(\ell+h^{\vee})} \left\{ \left\{  2( 2c(\mu) + 2c(1)) - c(2) - 2 c(\mu-1)  -2 c(\mu+1) \right\} \right\} \\
& =& 0.
\end{eqnarray*}
The other cases from Lemma \ref{mu rank lemma} where $r_{\vec{\mu}} >0$ can be checked similarly, yielding the result.
\end{proof}

\begin{remark}  
Necessary and sufficient conditions for an $\sL_2$ conformal blocks bundle to have 
$\operatorname{rk}(\mathbb{V}(\sL_2,\ell, \vec{\lambda})) \neq 0$ are given as follows.
\begin{lemma}[Swinarski, 2010] \label{r not zero}
Let $\mathfrak{g}=\sL_2$.  Then $r_{\vec{\lambda}} \neq 0$ if and only
$\Lambda = \sum_{i=1}^{n} \lambda_i$ is even, and for any subset $I \subseteq \{ 1,\ldots,n\}$ with $n-|I|$ odd, the inequality
\begin{equation}\label{gen tri ineq} 
\Lambda - (n-|I|-1) \ell \leq 2 \sum_{i \in I} \lambda_{i}
\end{equation}
is satisfied.
\end{lemma}
The main achievement of this lemma is in finding the correct statement; with this in hand, the result may be proved in a straightforward way using induction on $n$ and factorization.  We omit the proof, as we do not use this result in the sequel.
\end{remark}

\section{Intersecting the divisors $\Dl$ with F-curves}\label{intersectwithbasis}
In this section, in Theorem \ref{intersect}, we give a simple formula
for the intersection of the divisors $\Dl$ with a basis of $1$-cycles given by the first family of curves defined in Proposition \ref{C1}.  

\begin{definition}\label{rank}
Suppose that $n$ is even and put  $$r_{\ell}(j,t)
=\operatorname{rank} 
\mathbb{V}(\sL_2, \ell, ( {\underset{\text{j times}}{\underbrace{1,\ldots,1}}},t )).$$
\end{definition}

\begin{theorem}\label{intersect}
$\mathbb{D}^{\sL_2}_{\ell,(1,\ldots,1)} \cdot F_{n-i-2,i 1,1}=r_{\ell}(i,\ell) \cdot r_{\ell}(n-i-2,\ell)$.
\end{theorem}

\begin{proof} 
We write $P_{\ell} = \{0,1,\ldots,\ell\}$, and write $\vec{\mu} =(\mu_1,\mu_2,\mu_3,\mu_4)$ for elements of $P_{\ell}^{4}$.  We use \cite{Fakh}*{Prop. 2.5}  applied to $\mathbb{V}(\sL_2, \ell, (1,\ldots,1))$ and the symmetric $\operatorname{F}$-curve $F=F_{j_1,j_2,j_3,j_4}$, given
by a partition $n=j_1+j_2+j_3+j_4$. 

 Using the notation from Definition \ref{rank}, this says:
\begin{equation} \label{fakh 2.5} 
\deg (\mathbb{V}(\sL_2, \ell, (1,\ldots,1)) |_{F}) = \sum_{\vec{\mu} \in P_{\ell}^{4}} \deg \mathbb{V}(\sL_2, \ell, \vec{\mu}) \  \  \prod_{k=1}^{4} r_{\ell}(j_k, \mu_k).
\end{equation}

Recall that the two point fusion rules for $\sL_2$ imply that $r_{(a,b)}=0$ unless $a=b$, in which case $r_{(a,b)}=1$.  Since our F-curves have two 1's on the spine, the only nonzero summands in \ref{fakh 2.5} occur when $\mu_3 = \mu_4 =1$.  By Lemma \ref{mu lemma}, we have $\deg \mathbb{V}(\sL_2, \ell, \vec{\mu}) = 0$ if $\vec{\mu} \neq (\ell,\ell,1,1)$, and 1 otherwise.  The formula follows.
\end{proof}

We present an example, which suggests several corollaries to Theorem \ref{intersect}.

\begin{example}
Consider the matrix of intersection numbers $\mathbb{D}^{\sL_2}_{\ell,(1,\ldots,1)} \cdot F_{n-i-2, i,1,1}$ for $n=16$, where in the table
we put $\mathbb{D}^{\sL_2}_{\ell}$ for $\Dl$.

\begin{equation*} \label{n16 example}
\begin{array}{cccccccc}
 & \mathbb{D}^{\sL_2}_{1} & \mathbb{D}^{\sL_2}_{2} & \mathbb{D}^{\sL_2}_{3} & \mathbb{D}^{\sL_2}_{4} & \mathbb{D}^{\sL_2}_{5} & \mathbb{D}^{\sL_2}_{6} & \mathbb{D}^{\sL_2}_{7} \\
F_{1,1,1} & 1 & 0 & 0 & 0 & 0 & 0 & 0 \\
F_{1,1,2} & 0 & 32 & 0 & 0 & 0 & 0 & 0 \\
F_{1,1,3} & 1 & 0 & 55 & 0 & 0 & 0 & 0 \\
F_{1,1,4} & 0 & 32 & 0 & 40 & 0 & 0 & 0 \\
F_{1,1,5} & 1 & 0 & 63 & 0 & 19 & 0 & 0 \\
F_{1,1,6} & 0 & 32 & 0 & 52 & 0 & 6 & 0 \\
F_{1,1,7} & 1 & 0 & 64 & 0 & 25 & 0 & 1 \\ 
\end{array}
\end{equation*}
Note that this matrix has full rank.  This shows that the divisors   
are independent.  Moreover, since in all of the columns there are
curves that intersect the CB divisors in degree zero, this also shows
that the divisors lie on the boundary of the nef cone.  \end{example}

We now derive  six corollaries to Theorem \ref{intersect}.  The first of these corollaries describes the pattern of zeroes observed in the matrix of the example above.

\begin{corollary}[Vanishing intersecting numbers]\label{zeroint} \begin{enumerate}
\item If $i < \ell$, then $\mathbb{D}^{\sL_2}_{\ell,(1,\ldots,1)} \cdot F_{n-i-2, i,1,1}  =0$.
\item If $i\not\equiv \ell \bmod{2} $, then $\mathbb{D}^{\sL_2}_{\ell,(1,\ldots,1)} \cdot F_{n-i-2, i,1,1}  =0$.
\end{enumerate}
\end{corollary}
\begin{proof}
For $i < \ell$, use the Generalized Triangle Inequality (Lemma \ref{GeneralizedTriangleInequality}), and for $i \ge \ell$, use the Odd Sum Rule (Lemma  \ref{oddsumlemma}).
\end{proof}

In the next four corollaries, we find formulas for intersection
numbers of four of the divisors $\Dl$ (that is, formulas for the first
two and last two columns of the matrix shown in the example above).  First, we give a lemma computing certain ranks:

\begin{lemma} \label{ranks for int numbers of D1, D2}
\begin{enumerate}
\item Suppose $\ell=1$.  Then for $k \in \Z, k \geq 0$ we have $r_{1}(2k+1,1) = 1$.  
\item Suppose $\ell=2$.  Then for $k \in \Z, k \geq 1$ we have $r_{2}(2k,2) = 2^{k-1}$. 
\item Suppose $\ell=2$.  Then for $k \in \Z, k \geq 0$ we have $r_{2}(2k+1,1) = 2^{k}$. 
\end{enumerate}
\end{lemma}
\begin{proof}
We use induction on $k$ and factorization. 

For the first formula, by the two point fusion rules, $r_{1}(1,1)=1$.  So suppose the formula is true up to $k-1$.  Factorization and applying the Odd Sum Lemma yields
\[  r_{1}(2k+1,1) = r_{1}(2k,0) r_{1}(2,0) + r_{1}(2k,1) r_{1}(2,1) = r_{1}(2(k-1)+1,1)=1.
\]

For the second two formulas: we may check that $r_{2}(2,2) = 1$ and $r_{1}(1,1)=1$.  So suppose these two formulas work up to $k-1$.  Factorization and applying the Odd Sum Lemma yields
\[ r_{2}(2k,2) = \sum_{\mu=0}^{2} r_{2}(2k-1,\mu) r_{(1,2,\mu)} = r_{2}(2k-1,1) = r_{2}(2(k-1)+1,1) = 2^{k-1} 
\]
and
\[  r_{2}(2k+1,1) = r_{2}(2(k-1)+1,1) r_{(1,1,0)} + r_{2}(2k,2) r_{(1,1,2)}  = 2^{k-1} + 2^{k-1} = 2^k.
\]
\end{proof}


\begin{corollary}\label{1}
\begin{equation}
\mathbb{D}^{\sL_2}_{1,(1,\ldots,1)} \cdot F_{a,b,c,d}  = \left\{\begin{matrix}
  1 & \mbox{ $abcd$ odd;}\\
  0 &  \mbox{ $abcd$ even.} 
\end{matrix} \right.
\end{equation}
\end{corollary}

\begin{proof}
Suppose first that $abcd$ is even.  Then at least one of the four integers, say $a$, is even.  Then when we apply formula (\ref{fakh 2.5}) to compute $\mathbb{D}^{\sL_2}_{1,(1,\ldots,1)} \cdot F_{a,b,c,d}$, to get a nonzero summand, we must have $\mu_1 = 0$ to have $r_{\ell}(a,\mu_1) \neq 0$.  (By the Odd Sum Lemma, we need $\mu_1$ even, but $P_{\ell} = \{0,1\}$ since $\ell=1$.)  Since $\mu_1=0$, by propagation, we know $\mathbb{V}(\sL_2, \ell, \vec{\mu})$ is a pullback from $\ov{M}_{0,3} = pt$.  Hence $\deg \mathbb{V}(\sL_2, 1, \vec{\mu}) = 0$.   

Now suppose that $abcd$ is odd.  Then the only nonzero summand in formula (\ref{fakh 2.5}) occurs when $\vec{\mu} = (1,1,1,1)$.  We can compute $\deg \mathbb{V}(\sL_2, 1, (1,1,1,1)) = 1$, and by Lemma \ref{ranks for int numbers of D1, D2} above, $r_{1}(a, 1)r_{1}(b, 1)r_{1}(c, 1)r_{1}(d, 1)=1$.  
\end{proof}

\begin{corollary}\label{2}
\begin{equation}
\mathbb{D}^{\sL_2}_{2,(1,\ldots,1)} \cdot F_{a,b,c,d}  = \left\{\begin{matrix}
  0 & \mbox{ $abcd$ odd;}\\
  2^{g-2} &  \mbox{ $abcd$ even.} 
\end{matrix} \right.
\end{equation}
\end{corollary}
\begin{proof}
Suppose that $abcd$ is odd.  If any $\mu_i$ is even, then $r_{\ell}(a,\mu_i) = 0$ by the Odd Sum Lemma.  We have $P_{\ell} = \{0,1,2\}$ since $\ell=2$, so we only possibly get a nonzero summand in formula (\ref{fakh 2.5}) when $\vec{\mu} = (1,1,1,1)$.  We can compute $\deg \mathbb{V}(\sL_2, 2, (1,1,1,1)) = 0$, so in fact this summand is zero, too.  

Suppose next that $a,b$ are even while $c,d$ are odd.  To get a nonzero summand in formula (\ref{fakh 2.5}) we must have $\mu_1$ and $\mu_2$ even and $\mu_3$ and $\mu_4$ odd.  However, if $\mu_1$ or $\mu_2$ is zero, then by propagation, we know $\mathbb{V}(\sL_2, \ell, \vec{\mu})$ is a pullback from $\ov{M}_{0,3} = pt$, and hence $\deg \mathbb{V}(\sL_2, 1, \vec{\mu}) = 0$.  Thus, we only get a nonzero summand in formula (\ref{fakh 2.5}) when $\vec{\mu} = (2,2,1,1)$.  We compute $\deg \mathbb{V}(\sL_2, 2, (2,2,1,1)) = 1$, and use Lemma \ref{ranks for int numbers of D1, D2} to show that 
\begin{displaymath}
r_{2}(a,2) r_{2}(b,2) r_{2}(c,2) r_{2}(d,2)  =  2^{\frac{a}{2}-1} 2^{\frac{b}{2}-1} 2^{\frac{c-1}{2}} 2^{\frac{d-1}{2}}  = 2^{\frac{a+b+c+d}{2}-3} = 2^{g-2}.
\end{displaymath}

Finally suppose that $a,b,c,d$ are all even.  As above, we may argue that we only get a nonzero summand in formula (\ref{fakh 2.5}) when $\vec{\mu} = (2,2,2,2)$, and then $\deg \mathbb{V}(\sL_2, 2, (2,2,2,2)) = 2$.  We then use Lemma \ref{ranks for int numbers of D1, D2}.ii to show that 
\begin{displaymath}
2 r_{2}(a,2) r_{2}(b,2) r_{2}(c,2) r_{2}(d,2)  =  2 \cdot 2^{\frac{a}{2}-1} 2^{\frac{b}{2}-1} 2^{\frac{c}{2}-1} 2^{\frac{d}{2}-1}  = 2^{g-2}.
\end{displaymath}
\end{proof}

\begin{lemma} \label{ranks for g-1}
\begin{enumerate}
\item $r_{\ell}(k,k)=1$, for all $1 \le k \le \ell$. 
\item $r_{\ell}(k,k-2)=k-1$, for all $2 \le k \le \ell+1$. 
\item $r_{\ell}(\ell,\ell+2)=\ell$.
\end{enumerate}
\end{lemma}
\begin{proof}
For the first statement, we use induction on $k$ with base case $k=1$.
  Indeed, $r_{\ell}(1,1)=1$ by the $2$-point fusion rule.   Assume $r_{\ell}(j,j)=1$ for $j<k$ and apply factorization
  with the partition $1^{k-1}|(1,k)$ to get:
  $$r_{\ell}(k,k)=\sum_{0 \le \mu \le \ell} r_{\ell}(k-1,\mu)r_{(1,k,\mu)}.$$
  By the three point fusion rules, we have $r_{(1,k,\mu)}=0$ if $\mu <
  k-1$, or if $\mu>k+1$, or if $\mu=k$.  We also have $r_{(1,k,\mu)}=1$ if $\mu=k-1$ or $\mu=k+1$.  However, if
  $\mu=k+1$, then $r_\ell(k-1,k+1)=0$ by the Generalized Triangle
  Inequality (Lemma \ref{GeneralizedTriangleInequality}).  So the only  nonzero summand in $r_{\ell}(k,k)$ is $r_{\ell}(k-1,k-1)$, which is
  1 by the induction hypothesis, and so we are done.

For the second statement, we use induction on $k$ with base case $k=2$.
When $k=2$, the statement is that $r_{\ell}(1,1,0)=1$, which is true
by propagation and the two point fusion rules.  So assume
$r_{\ell}(j,j-2)=j-1$ for  $2 \leq j \leq k-1$.   Apply factorization to $1^{k-1} \cup (1,k-2)$.
    $$r_{\ell}(k,k-2)=\sum_{0 \le \mu \le \ell} r_{\ell}(k-1,\mu)r_{(1,k-2,\mu)}.$$
As before, by the three point fusion rules, we have $r_{(1,k-2,\mu)}$
if $\mu < k-3$, or if $\mu > k-1$, or if $\mu=k-2$.  We also have
$r_{(1,k-2,\mu)}=1$ if $\mu=k-3$ or if $\mu=k-1$.  Thus there are only
two nonzero summands in $r_{\ell}(k,k-2)$:
\begin{displaymath}
r_{\ell}(k,k-2) = r_{\ell}(k-1,k-3) + r_{\ell}(k-1,k-1).
\end{displaymath}
By the induction hypothesis, we have $r_{\ell}(k-1,k-3)=k-2$, and
by the first statement of this lemma, we have $r_{\ell}(k-1,k-1)=1$.
Thus $r_{\ell}(k,k-2) =k-1$, as claimed.

For the third statement, we apply factorization using the partition
$1^{\ell+1} \cup (1,\ell)$:
  $$r_{\ell}(\ell+2,\ell)=\sum_{0 \le \mu \le \ell} r_{\ell}(\ell+1,\mu)r_{(1,\ell,\mu)}.$$
We can argue as we did above that there is only one nonzero summand,
and it occurs for $\mu=\ell-1$.  Thus $r_{\ell}(\ell+2,\ell) =
r_{\ell}(\ell+1,\ell-1)$, and by the second statement, this is $\ell$.  

\end{proof}

\begin{corollary}\label{g-1}
\begin{equation}
\mathbb{D}^{\sL_2}_{g-1,(1,\ldots,1)} \cdot F_{n-i-2, i,1,1}  = \left\{\begin{matrix}
  0 & i \ne g-1\\
g-1&  i =g-1.
\end{matrix} \right.
\end{equation}
\end{corollary}
\begin{proof}
If $i \le g-2$, then by Corollary \ref{zeroint}, we have $\mathbb{D}^{\sL_2}_{g-1,(1,\ldots,1)} \cdot F_{n-i-2, i,1,1}=0$.
And for $i=g$, by Theorem \ref{intersect}
$$\mathbb{D}^{\sL_2}_{g-1,(1,\ldots,1)} \cdot F_{g,g,1,1}=r_{g-1}(g,g-1) \cdot r_{g-1}(g,g-1).$$
As $2g-1$ is odd, by the Odd Sum Lemma \ref{oddsumlemma}, $r_{g-1}(g,g-1)=0$
This leaves $i=g-1$. Again by Theorem \ref{intersect}
$$\mathbb{D}^{\sL_2}_{g-1,(1,\ldots,1)} \cdot F_{g+1,g-1,1,1}=r_{g-1}(g-1,g-1) \cdot r_{g-1}(g+1,g-1).$$

By the first statement of Lemma \ref{ranks for g-1} we have
$r_{g-1}(g-1,g-1) =1$, and by the third statement we have
$r_{g-1}(g+1,g-1)=g-1$.  The result follows.  
 \end{proof}

\begin{corollary}\label{g}
\begin{equation}
\mathbb{D}^{\sL_2}_{g,(1,\ldots,1)} \cdot F_{n-i-2, i,1,1}  = \left\{\begin{matrix}
  0 & i \le g-1\\
  1 &  i =g.
\end{matrix} \right.
\end{equation}
\end{corollary}

\begin{proof}
If $i \le g-1$, then by Corollary \ref{zeroint}, we have $\mathbb{D}^{\sL_2}_{g,(1,\ldots,1)} \cdot F_{n-i-2, i,1,1}=0$.
And for $i=g$, by Theorem \ref{intersect}
$$\mathbb{D}^{\sL_2}_{g,(1,\ldots,1)} \cdot F_{g,g,1,1}=r_{g}(g,g) \cdot r_{g}(g,g).$$

By the first statement of Lemma \ref{ranks for g-1}, we have
$r_{g}(g,g)=1$.
\end{proof}

\begin{corollary} \label{basis}
 $\{\Dl : 1 \le \ell \le g\}$
is a basis for $\Pic(\ov{M}_{0,2g+2}/S_{2g+2})$. 
\end{corollary}
\begin{proof}  The matrix of intersection numbers between these divisors and the F-curves $F_{n-i-2, i,1,1}$ is lower triangular with nonzero entries on the diagonal, and so the divisors are linearly independent.
To see this, note that by the Generalized Triangle Inequality, Lemma\label{Generalized Triangle Inequality}, one has $r_{\ell}(i,\ell)=0$ if $i<l$.  Thus, the entries above the diagonal are all zero.  On the other hand, $r_{\ell}(\ell,\ell)=1$.  One can apply factorization to show that $r_{\ell}(n-i-2,\ell) \neq 0$ as well, by partitioning the weight vector $1^{n-i-2} \ell$ as $1^{i} \ell | 1^{n-2i-2}$.  As the rank of $\Pic(\ov{M}_{0,n}/S_n)$ is $g$, the result follows.
\end{proof}

\section{Classes of the $\Dl$}\label{classes}
 In this Section we prove Proposition \ref{reducedformula}, which
gives a general formula for the classes of the divisors 
$\mathbb{D}^{\sL_2}_{\ell,(1,\ldots,1)}$.  This expression depends on 
ranks of certain $\sL_2$ CB-bundles.  We give these ranks in
Proposition \ref{rlclosedform}, which is proved in Section \ref{generatingsubsection}.  In Section \ref{extremalformulae} we give simplified expressions  for the four elements of the family that
generate extremal rays and we give simplified versions of the classes of some of the divisors
that lie on higher dimensional extremal faces.

Recall that 
for $0 \le t \le \ell$
$$r_{\ell}(i,t)
=\operatorname{rank}(
\mathbb{V}(\sL_2, \ell, ( {\underset{\text{i times}}{\underbrace{1,\ldots,1}}},t ))).$$

We will put $r_{\ell}(i,t)= 0$, for $t<0$, and for $t> \ell$, and write
$$r_{\ell}(n)
=\operatorname{rank}(
\mathbb{V}(\sL_2, \ell, (1,\ldots,1 ))).$$

\begin{proposition}\label{reducedformula} 
$ \Dl = \frac{1}{2(\ell + 2)} \sum_{i= 2}^{g+1} \left[ 
\frac{i(n-i)}{n-1} \beta_1 - \beta_i\right] B_{i}$,
where 
$$\beta_1=\frac{3}{2}r_{\ell}(n), \ \text{and   }   \ \beta_i=\sum_{t=0}^{\ell} (\frac{t^2}{2}+t) r_{\ell}(i,t) r_{\ell}(n-i, t) \ \  \text{ for  } 2 \le i \le g+1.$$
\end{proposition}
\begin{proof}
We obtain this formula from \cite{Fakh}*{Cor. 3.5}, using that
on $\sL_2$, one has $h^{\vee}=2$ and $t^*=t$ and $c(t)=\frac{t^2}{2}+t$.
\end{proof}

We give an explicit formula for the ranks $r_{\ell}(j,t)$ in the next proposition.  

\begin{proposition} \label{rlclosedform}  Let $j \geq 0$, and let $0 \leq t \leq \ell$.  Write $K:= \lceil j/(2(\ell+2)  \rceil$.  Then 
\begin{enumerate}
\item If $j,t$ are both even, say $j=2x$ and $t=2y$, then 
\begin{equation}\label{A}
r_{\ell}(j,t)  = \sum_{k=0}^{K}  \left(  b_k  \ \binom{2x}{x-y-k(\ell+2)}   - c_k \  \binom{2x}{x-(k+1)(\ell+2)+y+1}  \right),
\end{equation}
where $b_k = \frac{2y+2k(\ell+2)+1}{x+y+k(\ell+2)+1}$ and $c_k=\frac{(2k+2)(\ell+2)-2y-1}{x+(k+1)(\ell+2)-y}$; and 
\item If $j,t$ are both odd, say $j=2x+1$ and $t=2y+1$, then 
\begin{equation}\label{B}
r_{\ell}(j,t) =  \sum_{k=0}^{K}   \left( b_k \   \binom{2x+1}{x-y-k(\ell+2)}   -  c_k  \binom{2x+1}{ x-(k+1)(\ell+2) +y+2}       \right),
\end{equation}
where $b_k=\frac{2y+2k(\ell+2)+2}{x+y+k(\ell+2)+2}$ and $c_k=\frac{2(k+1)(\ell+2)-2y-2}{x+(k+1)(\ell+2) - y}$.
\item If $j$ and $t$ have opposite parity, then $r_{\ell}(j,t)=0$. 
\end{enumerate}
\end{proposition}

Proposition \ref{rlclosedform} follows from Propositions \ref{infinitysolution}
\ref{relatethem}.  These are both  proved in the following subsection.

\begin{note}
Compare the Equations \ref{A} and \ref{B} above to the Verlinde formula (\ref{Verlinde
  formula}).  The formulas of Proposition \ref{rlclosedform} appear more
complicated, but can be evaluated using only arithmetic operations (no trigonometric functions are required).  Thus, it is clear from
Proposition \ref{rlclosedform} that the ranks are rational numbers,
which is not obvious in the Verlinde formula.  In fact, with just a little
more work, it is easy to argue that the numbers $r_{\ell}(j,t)$ in
Proposition \ref{rlclosedform} are in fact integers.
\end{note}

\subsection{Proofs of the rank formulas in Proposition \ref{rlclosedform}}\label{generatingsubsection}

The proof of Proposition \ref{rlclosedform} involves three steps.
First, in Proposition \ref{rankrecurrences} we show the $r_{\ell}(j,t)$ are determined by a system of recurrences.
Second, in Definition \ref{infinity} and Proposition \ref{infinitysolution} we define a system of recurrences and solve it to get
an array of numbers $r_{\infty}(j,t)$.
Last, in Proposition \ref{relatethem} we explicitly relate the $\{r_{\ell}(j,t)\}$ and the $\{r_{\infty}(j,t)\}$.

\begin{proposition}\label{rankrecurrences}
The ranks $r_{\ell}(j,t)$ are determined by the system of recurrences
\begin{equation}\label{Pascal}  r_{\ell}(j,t)  =   r_{\ell}(j-1,t-1) + r_{\ell}(j-1,t+1),  \qquad t = 1,\ldots,\ell.
\end{equation}
together with seeds 
$$r_{\ell}(j,j)=1, \ \mbox{ if  } j \le \ell, \ \ \mbox{and } \ \ r_{\ell}(j,j)=0, \ \mbox{ if  } j > \ell.$$
\end{proposition}

\begin{note} We observe that (\ref{Pascal}) is somewhat reminiscent of the recurrence for Pascal's triangle.
\end{note}

\begin{proof}
Partition the weight vector $(1,\ldots,1,t)=1^{j}t$ as $1^{j-1} \cup (1,t)$.  If $j +t$ is odd, then by the Odd Sum Rule,  Lemma \ref{oddsumlemma}, $r_{\ell}(j,t) =0$.  So assume $j+t$ is even.  Then the factorization formula states 
\begin{equation} 
\label{factorization} r_{\ell}(j,t) = \sum_{\mu =0}^{\ell} r_{(1^{j-1} \cup \mu)} r_{(1, t,  \mu) }.
\end{equation}
We can simplify this expression.  Recall that by the $\sL_2$ fusion
rules (Prop. \ref{sl2 fusion rules}), $r_{(1, t, \mu)}$ is 0 if $\mu
>t+1$ or if $\mu < t-1$.  Thus the only possibly nonzero summands in (\ref{factorization}) are when $\mu = t-1$, $t$, or $t+1$.  But when $\mu = t$, by the Odd Sum Rule, Lemma \ref{oddsumlemma}, we have $r_{(1,t,t)} =0$.  Thus (\ref{factorization}) simplifies to the following:
\begin{eqnarray}   
r_{\ell}(j,t) & = &  r_{\ell}(j-1,t-1) + r_{\ell}(j-1,t+1) \qquad \qquad t = 1,\ldots,\ell-1;\\
r_{\ell}(j,\ell) & = & r_{\ell}(j-1,\ell-1).
\end{eqnarray}
Since $r_{\ell}(j-1,\ell+1) =0$, we can unify the two lines above, yielding (\ref{Pascal}).

\end{proof}

\begin{definition}\label{infinity}
Let $r_{\infty}(j,t)$ be the solutions of the system of recurrences 
\[ r_{\infty}(j,t)  =   r_{\infty}(j-1,t-1) + r_{\infty}(j-1,t+1)
\]
with seeds  \begin{enumerate}
\item $r_{\infty}(j,t) =0$ if $t>j$;
\item $r_{\infty}(j,-1) =0$ for all $j$; and
\item $r_{\infty}(j,j) = 1$ for all $j$.
\end{enumerate}
\end{definition}

\begin{proposition}\label{infinitysolution}[Values of $r_{\infty}(j,t)$] Suppose $j \geq 0$ and $ 0 \leq t \leq j$.
\begin{enumerate}
\item If $j$ and $t$ are both even, say $j = 2x$ and $t = 2y$, then $r_{\infty}(j,t) = \frac{2y+1}{x+y+1} \binom{2x}{x-y}$.
\item If $j$ and $t$ are both odd, say $j = 2x+1$ and $t = 2y+1$, then $r_{\infty}(j,t) = \frac{2y+2}{x+y+2} \binom{2x+1}{x-y}$.
\item If $j \not\equiv t \pmod{2}$ then $r_{\infty}(j,t) = 0$.
\end{enumerate}
\end{proposition}
\begin{proof} Straightforward check.
\end{proof}
\begin{remark}  For $t=1$, the formulas above give $\frac{1}{x+1} \binom{2x}{x}$.  These are the well-known Catalan numbers.  
\end{remark}

Next we relate the numbers $\{ r_{\infty}(j,t) \}$ and $\{ r_{\ell}(j,t) \}$.

\begin{proposition} \label{relatethem} 
Let $j \geq 0$, and let $0 \leq t \leq \ell$.  Write $K:= \lceil j/(2(\ell+2)  \rceil$.  Then 
\begin{equation}  
r_{\ell}(j,t)  = \sum_{k=0}^{K}   \left( r_{\infty} \left( \rule{0pt}{12pt} j, t + 2(\ell+2)k \right)   - r_{\infty} \left( \rule{0pt}{12pt} j, (2k+2)(\ell+2)-t-2 \right)  \right).
\end{equation}
\end{proposition}

In the formulas above, we use the convention that $\binom{n}{k} =0$ if $k < 0$.   

\begin{corollary} Fix $j,t$.  For $\ell$ sufficiently large with respect to $j,t$, we have $r_{\ell}(j,t) = r_{\infty}(j,t)$.
\end{corollary}
\begin{proof} When $\ell$ is large, there are no nontrivial reflections in the algorithm presented below.
\end{proof}

\begin{proof}[Proof of Proposition \ref{relatethem}]

Suppose that $j_{0} \geq 0$, and $0 \leq t_{0} \leq \ell$.  
To obtain $r_{\ell}(j_0,t_0)$ from the numbers $\{ r_{\infty}(j,t) \}$:
\begin{enumerate}
\item Mark every $(\ell+2^{th})$ column in the array of numbers $\{ r_{\infty}(j,t) \}$.  That is, starting with $k=0$, while $(k+1)(\ell+2)-1 \leq j$, mark the columns $t = (k+1)(\ell+2)-1$.
\item Find the successive reflections of $r_{\infty}(j_0,t_0)$ across these marked columns.
\item $r_{\ell}(j_0,t_0)$ is the alternating sum of these reflections.
\end{enumerate}

To prove this algorithm works, we use induction on $j$.
For $j =0$, we have $r_{\infty}(0,0) = 1$ and $r_{\infty}(0,t)=0$ for $t >0$, so the algorithm above is correct.
Suppose the algorithm works for all rows up to row $j-1$.  We check that it works for row $j$ as well: run the algorithm on $r_{\ell}(j,t)$ to get $r_{\ell}(j,t)$ as an alternating sum of $r_{\infty}(j,t)$.  Apply the recursion
 $r_{\infty}(j,t) = r_{\infty}(j-1,t-1) + r_{\infty}(j-1,t+1)$ to each term in the alternating sum.  We can regroup the resulting terms and apply the induction hypothesis to write this expression as $r_{\ell}(j-1,t-1) + r_{\ell}(j-1,t+1)$.  Thus the algorithm produces numbers that satisfy the correct recurrences.

The algorithm doesn't apply to the zero seed values.  (That is, we need to define these values separately, so we may set them to zero as desired.)  We only have to check the following seeds: for $j \leq \ell$, there are no nontrivial reflections, and we have $r_{\ell}(j,j) = r_{\infty}(j,j)=1$.  
\end{proof}

Proposition \ref{rlclosedform} now follows from Propositions \ref{infinitysolution} and 
\ref{relatethem}.


\begin{example}
By using propagation, factorization, and recursion, we compute $r_{3}(15,3) = 377$.
We compute $r_{3}(15,3)$ using the proposition.  A portion of the
matrix $r_{\infty}(i,j)$ is shown in the table below, and
we see that  $r_{\infty}(15,3) - r_{\infty}(15,5) + r_{\infty}(15,13)
- r_{\infty}(15,15) = 2002- 1638+14-1 = 377.$
\end{example}

\small
\begin{center}
\begin{tabular}{m{0.4cm}|m{0.65cm}m{0.65cm}m{0.65cm}m{0.65cm}|m{0.65cm}|m{0.65cm}m{0.65cm}m{0.65cm}m{0.65cm}|m{0.65cm}|m{0.65cm}m{0.65cm}m{0.65cm}m{0.65cm}|m{0.65cm}|m{0.65cm}}
   & 0 & 1 & 2 & 3  &4 & 5 & 6& 7 & 8 & 9 & 10 & 11 & 12 & 13 & 14 & 15 \\ \hline 
0 & 1 &  &  &   & &  & &  &  &  &  &  & &  & &  \\
1 & 0 & 1 &  &   & &  & &  &  &  &  &  & &  & &  \\ 
2 & 1 & 0 & 1 &   & &  & &  &  &  &  &  & &  & &  \\ 
3 & 0 & 2 & 0 & 1 & &  & &  &  &  &  &  & &  & &  \\ 
4 & 2 & 0 & 3 & 0 & 1 &  & &  &  &  &  &  & &  & &  \\    
5 & 0 & 5 & 0 & 4 & 0 & 1 & &  &  &  &  &  & &  & &  \\
6 & 5 & 0 & 9 & 0 & 5 & 0 & 1 &  &  &  &  &  & &  & &  \\    
7 & 0 & 14 & 0 & 14 & 0 & 6 & 0 & 1 &  &  &  &  & &  & &  \\
8 & 14 & 0 & 28 & 0 & 20 & 0 & 7 & 0 & 1 &  &  &  & &  & &  \\    
9 & 0 & 42 & 0 & 48 & 0 & 27 & 0 & 8 & 0 & 1 &  &  & &  & &  \\
10 & 42 & 0 & 90 & 0 & 75 & 0 & 35 & 0 & 9 & 0 & 1 &  & &  & &  \\    
11 & 0 & 132 & 0 & 165 & 0 & 110 & 0 & 44 & 0 & 10 & 0 & 1 & &  & &  \\
12 & 132 & 0 & 297 & 0 & 275 & 0 & 154 & 0 & 54 & 0 & 11 & 0 & 1 &  & &  \\    
13 & 0 & 429 & 0 & 572 & 0 & 429 & 0 & 208 & 0 & 65 & 0 & 12 & 0 & 1 & &  \\
14 & 429 & 0 & 1001 & 0 & 1001 & 0 & 637 & 0 & 273 & 0 & 77 & 0 & 13 & 0 & 1&  \\    
15 & 0 & 1430 & 0 & \color{green}{2002} & 0 & \color{red}{1638} & 0 & 910 & 0 & 350 & 0 & 90 & 0 & \color{green}{14} & 0 & \color{red}{1} 
\end{tabular}
\end{center}
\normalsize

\subsection{Simplified versions of the four extremal rays and four other extremal divisors}\label{extremalformulae}
In Section \ref{extremal} we show that for $\ell \in \{1,2,g-1,g\}$ the divisors $\Dl$ generate extremal rays
of the symmetric nef cone.  We give expressions for their divisor
classes now.

\begin{theorem}\label{extremaldivisors}
\begin{eqnarray*}
\mathbb{D}^{\sL_2}_{1,(1,\ldots,1)} & = &   \sum_{2 \leq k \leq g+1, k \, even} \frac{k(n-k)}{4(n-1)} B_{k}  + \sum_{2 \leq k \leq g+1, k \, odd} \frac{(k-1)(n-k-1)}{4(n-1)} B_{k}; \\
\mathbb{D}^{\sL_2}_{2,(1,\ldots,1)}  & = &
3\cdot2^{g-1} \left(\sum_{2 \leq k \leq g+1, k \, even} \left(\frac{k(n-k)}{8(n-1)} - \frac{1}{6} \right) B_{k}  + \sum_{2 \leq k \leq g+1, k \, odd}  \frac{(k-1)(n-k-1)}{8(n-1)} B_{k}\right); \\
\mathbb{D}^{\sL_2}_{g-1,(1,\ldots,1)}   & = &
(g-1)\left(\sum_{k=2}^{g}  \frac{(k-1)k}{(n-1)} B_{k}   + \left(
    \frac{g^2-g-1}{(n-1)} \right)B_{g+1} \right); \\
\mathbb{D}^{\sL_2}_{g,(1,\ldots,1)}  & = & 2\sum_{k=2}^{g+1}  \frac{(k-1)k}{(n-1)} B_{k}.
\end{eqnarray*}

\end{theorem}

\begin{proof}
We intersect
the expressions above with the basis of $1$-cycles $\{F_{1,1,i}:  1\le
i \le g\}$ and check that we get the formulas given in Corollaries \ref{1}, \ref{2}, \ref{g-1}, and \ref{g}.  

Note the simple but useful formula which follows from \cite{KM}:
$$ \sum_{k=2}^{n/2} c_k B_k \cdot F_{1,1,i}=-c_{i+2}-c_i+c_2+2c_{i+1}.$$

We consider $\mathbb{D}^{\sL_2}_{1,(1,\ldots,1)}$ first.  Multiplying by the constant $4(n-1)$ one has for $i$ even, that
$$4(n-1)\mathbb{D}^{\sL_2}_{1,\{1,\ldots,\}} \cdot F_{1,1,i}=-(i+2)(n-(i+2))-i(n-i)+2(n-2)+2i(n-i-2)=0,$$
and for $i$ odd,
$$4(n-1)\mathbb{D}^{\sL_2}_{1,\{1,\ldots,\}} \cdot F_{1,1,i}=-(i+1)(n-i-3)-(i-1)(n-i-1)+2(n-2)+2(i+1)(n-(i+1))=4(n-1).$$
In other words,
\begin{equation}
\mathbb{D}^{\sL_2}_{1,(1,\ldots,1)} \cdot F_{n-i-2, i,1,1}  = \left\{\begin{matrix}
  1 & i \mbox{ odd;}\\
  0 &  i \mbox{ even,} 
\end{matrix} \right.
\end{equation}
Therefore, by Corollary \ref{basis} and Lemma \ref{1}, we have the correct expression for $\mathbb{D}^{\sL_2}_{1,(1,\ldots,1)}$.

Next, we consider $\mathbb{D}^{\sL_2}_{2,(1,\ldots,1)}$.    One checks easily that for $i$ even, 
$$\frac{1}{3}\cdot \frac{1}{2^{g-1}} \mathbb{D}^{\sL_2}_{2,(1,\ldots,1)}\cdot F_{1,1,i} =\frac{1}{6} \  \implies \ \mathbb{D}^{\sL_2}_{2,(1,\ldots,1)}\cdot F_{1,1,i} =2^{g-2}.$$
It is also straightforward to check that for $i$ odd, $\mathbb{D}^{\sL_2}_{2,(1,\ldots,1)}\cdot F_{1,1,i} =0$. 
In other words,
\begin{equation}
\mathbb{D}^{\sL_2}_{2,(1,\ldots,1)} \cdot F_{n-i-2, i,1,1}  = \left\{\begin{matrix}
  0 & i \mbox{ odd;}\\
  2^{g-2} &  i \mbox{ even,} 
\end{matrix} \right.
\end{equation} Therefore, by Corollary \ref{basis} and Lemma \ref{2}, we have the correct expression for $\mathbb{D}^{\sL_2}_{2,(1,\ldots,1)}$.

Next, we consider $\mathbb{D}^{\sL_2}_{g-1,(1,\ldots,1)}$.  For $i \le g-2$, 
$$\frac{(n-1)}{(g-1)}\mathbb{D}^{\sL_2}_{g-1,(1,\ldots,1)}\cdot F_{1,1,i}=-(i+1)(i+2)-(i-1)i+2+2i(i+1)=0.$$
For $i=g$, 
\begin{multline}
\frac{(n-1)}{(g-1)}\mathbb{D}^{\sL_2}_{g-1,(1,\ldots,1)}\cdot F_{1,1,g,g}=-2c_g+c_2+2c_{g+1}\\
=-2g(g-1)+2+2\frac{(g^3-2g^2+1)}{(g-1)}=-2g(g-1)+g(2g-2)=0.
\end{multline}%
Whereas, if $i=g-1$, then $i+1=g$ and $i+2=g+1$, and so
\begin{eqnarray*}\frac{(n-1)}{(g-1)}\mathbb{D}^{\sL_2}_{g-1,(1,\ldots,1)}\cdot F_{1,1,g-1}
& = & -\frac{(g^3-2g^2+1)}{(g-1)}-(g-2)(g-1)+2+2g(g-1) =2g+1 \\
\Rightarrow \mathbb{D}^{\sL_2}_{g-1,(1,\ldots,1)}\cdot F_{1,1,g-1} & =&  g-1
\end{eqnarray*}
\noindent
 Therefore, by Corollary \ref{basis} and Lemma \ref{g-1}, we have the correct expression for $\mathbb{D}^{\sL_2}_{g-1,(1,\ldots,1)}$.

Finally, we consider $\mathbb{D}^{\sL_2}_{g,(1,\ldots,1)}$.  For $i \le g-1$, 
$$2(n-1) \mathbb{D}^{\sL_2}_{g,(1,\ldots,1)}\cdot F_{1,1,i}=-(i+1)(i+2)-(i-1)i+2+2i(i+1)=0,$$
and for $i=g$,
$$2(n-1) \mathbb{D}^{\sL_2}_{g,(1,\ldots,1)}\cdot F_{1,1,g}=-2g(g-1)+2+2g(g+1)=2(n-1).$$
\noindent
Therefore, by Corollary \ref{basis} and Lemma \ref{g}, we have the correct expression for $\mathbb{D}^{\sL_2}_{g,(1,\ldots,1)}$.
\end{proof}


We have also found simple formulas for the classes of some of the
other divisors by ad hoc methods.  We list a few such formulas here.

\begin{example} $\mathbb{D}^{\sL_2}_{3,(1,\ldots,1)}$.  Let $F_n$ be the Fibonacci numbers given by $F_{0}=0, F_{1}=1$, and $F_{n} = F_{n-1} + F_{n-2}$ for $n \geq 2$.  Then $\mathbb{D}^{\sL_2}_{3,(1,\ldots,1)}=\sum b_k B_k$, where 
if $k=2j+1$ is odd, then
\[  b_{k} = \frac{1}{10}\left( \frac{3}{2 }\frac{k(n-k)}{n-1} F_{2g+1} - \frac{3}{2} F_{2j+1} F_{2(g-j)+1}  - \frac{15}{2} F_{2j} F_{2(g-j)} \right).
\] 
and if $k=2j$ is even,  then
\[  b_{k} = \frac{1}{10}\left( \frac{6j(g-j+1)}{2g+1} F_{2g+1} - 4 F_{2j} F_{2(g-j+1)} \right).
\] 
\end{example}

\begin{example} $\mathbb{D}^{\sL_2}_{4,(1,\ldots,1)} = \sum b_k B_k$, where if $k$ is odd, then
\begin{displaymath}b_k = \frac{1}{16} \left(  \frac{k(n-k)}{n-1} (3^{\frac{n-2}{2}} + 1) - \frac{1}{2}(3^{\frac{k-1}{2}}+1)(3^{\frac{n-k-1}{2}}+1) - \frac{5}{2}(3^{\frac{k-1}{2}}-1)(3^{\frac{n-k-1}{2}}-1) \right)
\end{displaymath}
and if $k$ is even, then
\begin{displaymath}b_k = \frac{1}{12} \left(  \frac{k(n-k)}{n-1} \frac{3}{4} (3^{\frac{n-2}{2}} + 1) - 4\cdot 3^{\frac{n-4}{2}}- 3(3^{\frac{k-2}{2}}-1)(3^{\frac{n-k-2}{2}}-1) \right)
\end{displaymath}
\end{example}

\begin{example}
$\mathbb{D}^{\sL_2}_{g-2,(1,\ldots,1)}= \sum b_k B_k$, where 
\begin{displaymath}
b_{k}  = \left\{
\begin{array}{ll}
\frac{4(n-7)(n-2)(k-1)k}{16(n-1)} & \mbox{ if $k \leq g-1$}\\
\frac{n^4-17n^3+90n^2-152n+96}{16(n-1)} & \mbox{ if $k =g $}\\
\frac{n^4-15n^3+60n^2-20n-32}{16(n-1)} & \mbox{ if $k =g+1$}.\\
\end{array}
\right.
\end{displaymath}
\end{example}

\section{Position of the divisors $\Dl$ in the nef cone}\label{position}

The divisors $\mathbb{D}^{\sL_2}_{\ell,(1,\ldots,1)}$ are semiample, and so
(suitably large multiples of) these divisors define morphisms on
$\ov{M}_{0,n}$.  Any $S_n$-invariant divisor lies on the interior of the cone
of effective divisors (\cite{KM} and \cite{GibneyCompositio}), and
since the $\mathbb{D}^{\sL_2}_{\ell,(1,\ldots,1)}$ are symmetric, the
morphisms they define are birational.   As the divisors are nef, we
can tell more about the morphisms they define by finding their
location in the nef cone.  Base point free divisors in the interior of
the nef cone give embeddings, and those on the boundary 
define contractions.

\subsection{The $\Dl$ are extremal}\label{extremal}
In this section we show that the morphisms given by the divisors $\mathbb{D}^{\sL_2}_{\ell,(1,\ldots,1)}$ on $\ov{M}_{0,n}/S_n$ are in fact birational contractions by proving the following.

\begin{theorem}  The divisors $\mathbb{D}^{\sL_2}_{\ell,(1,\ldots,1)}$, for $1 \le \ell \le g $
reside on the boundary of  $\Nef(\ov{M}_{0,n}/S_n)$, and
for $\ell \in \{1,2,g-1,g\}$ the divisors span extremal rays of the cone. 
In particular, the divisors $\mathbb{D}^{\sL_2}_{\ell,(1,\ldots,1)}$, for $1 \le \ell \le g$  define birational contractions on $\ov{M}_{0,n}/S_n$.
\end{theorem}

In Section \ref{classes} we gave the classes of the divisors $\mathbb{D}^{\sL_2}_{\ell,(1,\ldots,1)}$.  As an application, we identify the corresponding morphisms, which we discuss in detail in Section \ref{morphisms}.

\begin{proof}
We write $n=2(g+1)$ and recall that the space  $\Pic(\ov{M}_{0,n}/S_n)\otimes \mathbb{R}$ is $g$-dimensional.
Therefore, a divisor lies on a face of dimension $k$ if it intersects $g-k$
basis elements in degree zero.  Hence the divisor $\mathbb{D}^{\sL_2}_{\ell,\vec{\lambda}}$ lies on a face of codimension at least $\ell$ by Corollary \ref{zeroint} using the basis $\{F_{1,1,i}: 1\le i \le g\}$, given by Corollary \ref{basis}.
In particular by these Corollaries,  $\mathbb{D}^{\sL_2}_{g-1,\vec{\lambda}}$ and $\mathbb{D}^{\sL_2}_{g,\vec{\lambda}}$
generate extremal rays of the cone.
To show that $\mathbb{D}^{\sL_2}_{1,(1,\ldots,1)}$ and $\mathbb{D}^{\sL_2}_{2,(1,\ldots,1)}$ also generate extremal rays, we will
use the two other sets of independent curves $\mathscr{C}_2$ and $\mathscr{C}_3$ from Theorem \ref{3curvefamilies}.

First, we note that by \cite{Fakh}*{Prop. 5.2}, which we can apply since $\ell=1$, one has
  $\mathbb{D}^{\sL_2}_{1,(1,\ldots,1)} \cdot F_{a,b,c,d} = 0$ if $abcd \equiv 0 \bmod 2$.   In particular, $\mathbb{D}^{\sL_2}_{1,(1,\ldots,1)} \cdot F_{2,2,i,n-4-i} = 0$ for all $1 \le i \le g-1$.  Since these are the curves from $\mathscr{C}_2$, and since $\mathbb{D}^{\sL_2}_{1,(1,\ldots,1)}$ is nontrivial, we have that it spans an extremal ray of the symmetric nef cone.

Next, we show that  $\mathbb{D}^{\sL_2}_{2,(1,\ldots,1)}  \cdot F_{a,b,c,d} = 0$ if $abcd \equiv 1 \bmod 2$.  This implies $a,b,c,d$ are all odd. This is proved using \cite{Fakh}*{Prop. 2.5},
$$\mathbb{D}^{\sL_2}_{2,(1,\ldots,1)} \cdot F_{a,b,c,d}
=\sum_{\stackrel{\vec{\mu}=\{\mu_1,\ldots,\mu_4\}}{0 \le \mu_i \le 2}}
\deg(\mathbb{V}(\sL_2, 2, \vec{\mu}))r_{2}(a,\mu_1)r_{2}(b,\mu_2)r_{2}(c,\mu_3)r_{2}(d,\mu_4).$$
Since the level is 2, each $\mu_{i}$ can only be 0, 1, or 2.  If $\mu_{i}$ is even for any $i$, then the corresponding rank is $0$, e.g. if $\mu_{1}$ is even then $r_{a \mu_{1}} = 0 $.  It remains only to consider the case where $\mu_{i}=1$ for $i=1,\ldots,4$.  But an explicit calculation using \ref{fakh deg formula} shows that 
$\deg (\mathbb{V}(\sL_2, 2, (1,1,1,1)))= 0$.  Thus, there are no nonzero contributions to Fakhruddin's formula \ref{fakh 2.5}.
We show in Proposition \ref{oddcurves}  that for $g=2k$ or $g=2k-1$, the set $\mathscr{C}_3=\{F_{3,3,2i+1}: 0 \le i \le k-2\}\cup \{F_{1,1,2i+1}: 0 \le i \le k-1\}$ consists of independent curves.
Therefore $\mathbb{D}^{\sL_2}_{2,(1,\ldots,1)}$ is extremal in the
symmetric nef cone.

\end{proof}

\subsection{The $\Dl$ are not all log canonical}

The Ray Theorem of Keel and McKernan \cite{KM}*{Thm. 1.2} and its extension by Farkas and Gibney \cite{FarkasGibney}*{Thm. 4} gives a tool for detecting nef divisors in what we call the log canonical part of the cone.  

\begin{definition}
We say a divisor $D$ on $\ov{M}_{0,n}$ is \emph{log canonical} if $D$ may be expressed as an effective combination
$$D=c(K_{\ov{M}_{0,n}}+\sum_{I \subset \{1,\ldots,n\}}c_I\delta_I),$$
where $c$ and the $c_I$ are any nonnegative rational numbers such that that $0 \le c_I \le 1$ for all $I$.
\end{definition}

If a divisor $D$ is log canonical, and $D$  intersects  every
$\operatorname{F}$-curve nonnegatively, then $D$ is nef by the Ray
Theorem.  Moreover, $D$ is semiample by \cite{BCHM}. 

For symmetric divisors, we can say even more.  Suppose $D$ is $S_n$-equivariant.  Then if $D$ is also log canonical, it is actually possible to write $D$ in the form 
$$D = c( K_{\ov{M}_{0,n}}+\sum_{i=2}^{g+1} b_i  B_i),$$ where $c>0$ and $0 \leq b_i \leq 1$ for all $i$.
This can be easily proved by averaging the log canonical expression for $D$ over $S_n$.  We call a divisor of the form $c( K_{\ov{M}_{0,n}}+\sum_{i=2}^{g+1} b_i  B_i)$ with $0 \leq b_i \leq 1$ \emph{symmetrically log canonical}.

Some of the divisors $\Dl$ are of this form, but others are not.  We interpret the failure of these divisors to be symmetrically log canonical to mean that they are outside the part of $\Nef({\ov{M}_{0,n}})$ that can be understood combinatorially.  This motivated us to study the geometry behind them in more detail.

\begin{proposition}\label{nlc}
\begin{enumerate}
\item $\mathbb{D}^{\sL_2}_{1,(1,\ldots,1)}$ is symmetrically log canonical for $n=6,8,10$, but not for $n \geq 12$.
\item $\mathbb{D}^{\sL_2}_{2,(1,\ldots,1)}$ is symmetrically log canonical for all $n \geq 6$.
\item $\mathbb{D}^{\sL_2}_{g-1,(1,\ldots,1)}$ is symmetrically log canonical for $n=10,12,14$, but not for $n \geq 16$.
\item $\mathbb{D}^{\sL_2}_{g,(1,\ldots,1)}$ is symmetrically log canonical for $n=8,10,12$, but not for $n \geq 14$.
\end{enumerate}
\end{proposition}
\begin{proof}
The claims regarding the divisors $\Dl$ for small values of $n$ may be verified by direct calculations.

The following identity shows that $\mathbb{D}^{\sL_2}_{2,(1,\ldots,1)}$ is symmetrically log canonical for all $n$:
$$\frac{8}{3 \cdot 2^{g-1}}\mathbb{D}^{\sL_2}_{2,(1,\ldots,1)}=K_{\ov{M}_{0,n}}+\frac{2}{3}\sum_{i \text{ even}}B_i+\sum_{i \text{ odd}} B_i.$$

It is easy to show that for $n \geq 12$, $g$ odd, there is no triple $(c,b_2,b_g)$ making $\mathbb{D}^{\sL_2}_{1,(1,\ldots,1)}$ symmetrically log canonical, and $n \geq 12$, $g$ even, there is no triple $(c,b_2,b_{g+1})$ making $\mathbb{D}^{\sL_2}_{1,(1,\ldots,1)}$ symmetrically log canonical.  Similarly, it is easy to show that for $n \geq 16$, there is no triple $(c,b_3,b_g)$ making $\mathbb{D}^{\sL_2}_{g-1,(1,\ldots,1)}$ symmetrically log canonical, and for $n \geq 14$, there is no triple $(c,b_2,b_{g+1})$ making $\mathbb{D}^{\sL_2}_{g,(1,\ldots,1)}$ symmetrically log canonical.
\end{proof}

\section{Morphisms defined by the divisors $\Dl$}\label{morphisms}
In this section we consider the morphisms defined by the extremal divisors
and by divisors that lie on some faces spanned by them.


\subsection{Levels 1 and 2 and extended Torelli maps}
Let $h: \ov{M}_{0,2(g+1)} \longrightarrow \ov{M}_{g}$ be the morphism defined by taking $(C,\vec{p})\in \ov{M}_{0,2(g+1)}$
to the stable curve of genus $g$ obtained by taking a double cover of $C$ branched at the set marked points $\vec{p}=\{p_1,\ldots,p_n\}$.  In this section we will show that the divisor $\mathbb{D}^{\sL_2}_{1,(1,\ldots,1)}$ defines a morphism that factors through $h$.

The following formula for the pullback of a divisor on $\ov{M}_g$ along the map $h$ will be useful:

\begin{lemma}\label{hpullback}
Let $h: \ov{M}_{0,2(g+1)} \longrightarrow \ov{H}_g \subset \ov{M}_{g}$
 be the isomorphism onto the hyperelliptic locus in $\ov{M}_{g}$, and let $D=a\lambda-\sum_{i=0}^{\lfloor \frac{g}{2} \rfloor}b_i\delta_i$ be a divisor on $\ov{M}_{g}$.  Then
 $$h^*(D)=\sum_{\stackrel{2 \le k \le \lfloor \frac{g}{2} \rfloor}{\text{k even}}}(\frac{ak(n-k)}{8(n-1)}-2b_0)B_k
 +\sum_{\stackrel{2 \le k \le \lfloor \frac{g}{2} \rfloor}{\text{k odd}}}(\frac{a(k-1)(n-k-1)}{8(n-1)}-\frac{b_i}{2})B_k.$$
\end{lemma}

\begin{proof}
This follows from \cite{CornalbaHarris}*{p. 468-470 and Prop. 4.7}.
\end{proof}

The classical Torelli map 
$M_g \overset{t}{\longrightarrow} A_g,$
which takes a smooth curve $X$ of genus $g$ to its Jacobian,
extends to a regular map
   $$t^{\text{Sat}}:\ov{M}_{g} \longrightarrow \ov{A}_g^{\text{Sat}},$$
   where $\ov{A}_g^{\text{Sat}}$ is the Satake compactification of the moduli space $A_g$. 
   This morphism $t^{\text{Sat}}$ is given by the divisor $\lambda$ .
   In other words, $\lambda=(t^{\text{Sat}})^*(\Theta),$ where
   $\Theta$ is the theta  divisor on $\ov{A}_g^{\text{Sat}}$ \cite{Arakelov}.
\begin{theorem}\label{Satake}
The divisor $\mathbb{D}^{\sL_2}_{1,(1,\ldots,1)}$ defines the composition

$$\ov{M}_{0,2g+2}/S_{2g+2}\overset{h}{\longrightarrow} \ov{M}_{g}  \overset{\overline{t}^{Sat}}{\longrightarrow} \ov{A}_g^{Sat}.$$
\end{theorem}

\begin{proof}
To prove this, we use Lemma \ref{hpullback} and Theorem \ref{extremaldivisors} to show
$$\mathbb{D}^{\sL_2}_{1,(1,\ldots,1)}=2 h^*(\lambda).$$
Because $\lambda$ is the semiample divisor that defines the morphism from $\ov{M}_{g}$ to $\overset{\overline{t}^{Sat}}{\longrightarrow} \ov{A}_g^{Sat}$, the result follows.
\end{proof}

It is natural to wonder whether any of the other divisors $\Dl$ define morphisms that factor through $h$.  We present some evidence which suggests this may be true for 
$\mathbb{D}^{\sL_2}_{2,(1,\ldots,1)}$.

\begin{theorem}\label{Hyperelliptic}
Write
$$h: \ov{M}_{0,2g+2}/S_{2g+2} \rightarrow  \ov{H}_g \hookrightarrow \ov{M}_{g}.$$
Then
$$\mathbb{D}^{\sL_2}_{2,(1,\ldots,1)}=\frac{1}{2} h^*(12\lambda - \delta_0).$$
\end{theorem}

\begin{proof}
Use Lemma \ref{hpullback} and Theorem \ref{extremaldivisors}.
\end{proof} 

Since $\mathbb{D}^{\sL_2}_{2,(1,\ldots,1)}$ is a conformal blocks divisor (or, alternatively, because it is symmetrically log canonical), we know that $h^*(12\lambda - \delta_0)$ is semiample on $\ov{M}_{0,n}$.  However, we do not know whether $12\lambda - \delta_0$ is semiample on $\ov{M}_{g}$ in general:

\begin{question}\label{unknown} Is  $12\lambda - \delta_0$ semiample on $\ov{M}_{g}$ for all $g \geq 4$?  
\end{question}

It is known that $12\lambda - \delta_0$ is nef \cite{FaberThesis}*{Prop. 3.3}.  Also, Rulla shows that $12\lambda - \delta_0$ is base-point free
for $g=3$ \cite{Rulla}*{Prop. 2.3.6}.  If $12\lambda - \delta_0$ is base-point free for all $g$, then Theorem \ref{Hyperelliptic} shows that the morphism given by $\mathbb{D}^{\sL_2}_{2,(1,\ldots,1)}$ factors through the hyperelliptic locus, as in the case for $g=3$.

Let $X$ denote the image of the linear system $|12 \lambda - \delta_{0}|$.  Even for $g=3$, where we know this divisor defines a morphism, we don't know a modular interpretation or classical description of $X$.  As one can see in \cite{Rulla}*{Fig. 2.8}, $X$ corresponds to a wall in the effective cone of $\ov{M}_3$ that lies between the full dimensional chambers that correspond to $\ov{A}_{3}^{\operatorname{Vor}(2)}$ and $\ov{M}_{3}^{ps}$.  It seems a reasonable guess that there might be morphisms from $\vortwo$ and $\ov{M}_{g}^{ps}$
to $X$ that are small modifications.  Here $\ov{M}_{g}^{ps}$ stands for the moduli space of pseudostable curves
(cf. \cite{Schubert}, \cite{HyeonLee3}, \cite{HassettHyeonDivCon}).

\subsection{Levels $g$ and $g-1$ and points on a line}

In \cite{simpsonlogcanonical}, Matthew Simpson  identified certain log canonical models of $\ov{M}_{0,n}$ as Hassett weighted spaces $\ov{M}_{0,A}$ (\cite{HassettWeighted}).  This result has been extended and reproved several times (\cite{AlexeevSwinarski},\cite{FedorchukSmyth}, \cite{KiemMoon}).  We use the notation of Kiem and Moon \cite{KiemMoon}.  

Let $(\mathbb{P}^1)^n \dblq \SL(2)$ denote the GIT quotient with the symmetric linearization.

\begin{theorem}\label{GIT}  The extremal divisor $\mathbb{D}^{\sL_2}_{g,(1,\ldots,1)}$ defines a 
sequence of contractions through Hassett's moduli spaces $\ov{M}_{0,n\cdot \epsilon}$
of weighted pointed stable curves with symmetric weights $n\cdot \epsilon=\{\epsilon,\ldots,\epsilon\}$.
Namely, for $n=2g+2$

$$\ov{M}_{0,n}=\ov{M}_{0,n\cdot \epsilon_{g-1}} \rightarrow \ov{M}_{0,n\cdot \epsilon_{g-2}}
\rightarrow \cdots \rightarrow \ov{M}_{0,n\cdot \epsilon_{1}} \rightarrow
(\mathbb{P}^1)^n \dblq \SL(2),$$
where $\frac{1}{g+2-k} < \epsilon_k < \frac{1}{g+1-k}$.
\end{theorem}
\begin{proof}
Compare the intersection numbers of $\mathbb{D}^{\sL_2}_{g,(1,\ldots,1)}$ given in Cor. \ref{g} with those for the pullback of the distinguished ample line bundle on the GIT quotient given in  \cite{AlexeevSwinarski}.
\end{proof}

We have $A=n \cdot \epsilon_1 = (\frac{2}{n-2}, \ldots, \frac{2}{n-2})$.  Kiem and Moon identify $\ov{M}_{0,n\cdot \epsilon_{1}}$ as Kirwan's partial desingularization of $(\mathbb{P}^1)^n \dblq SL(2)$.  They give a description of $\Pic( \ov{M}_{0,n\cdot \epsilon_{k}}  )$ in \cite{KiemMoon}*{Thm 6.1.2}, from which it follows that $\dim \NS( \ov{M}_{0,n\cdot \epsilon_{1}} / S_{n} ) =2$.  Alternatively, one can easily check the following claim directly:

\begin{claim}
Let $\rho_{n \cdot \epsilon_1}: \ov{M}_{0,n} \rightarrow \ov{M}_{0,n\cdot \epsilon_{1}}$ be the birational contraction defined by Hassett.
\begin{enumerate}
\item $\rho_{n \cdot \epsilon_1}$ contracts $\Delta_{I}$ if $|I| \neq 2, \frac{n}{2}$.
\item $\rho_{n \cdot \epsilon_1}$ contracts $F_{k,1,1}$ if $k \leq g-2$.  
\end{enumerate}
\end{claim}

By Theorem \ref{GIT} above, we know that the morphism associated to $\mathbb{D}^{\sL_2}_{g,(1,\ldots,1)}$ factors through $\ov{M}_{0,n\cdot \epsilon_{1}}$, and by \cite{Fakh}*{Prop. 4.7}, we know that the morphism associated to $\mathbb{D}^{\sL_2}_{g-1,(1,\ldots,1)}$ also factors through $\ov{M}_{0,n\cdot \epsilon_{1}}$.  We combine this with the intersection numbers given in Corollaries \ref{g-1} and \ref{g} to obtain the following result: 
\begin{theorem} \label{Dg-1morphism}
Let $Y_{g-1}$ denote the image of the morphism associated to the divisor $\mathbb{D}^{\sL_2}_{g-1,(1,\ldots,1)}$.  Then $Y_{g-1}$ is the (generalized) flip of $(\mathbb{P}^1)^n \dblq SL(2) $ which contracts the curve $F_{g,1,1}$ in $\ov{M}_{0,n\cdot \epsilon_{1}}$:
\begin{displaymath}
\xymatrix{   &  \ov{M}_{0,n} \ar[d]^{\rho_{n \cdot \epsilon_1}} & \\
& \ov{M}_{0,n\cdot \epsilon_{1}}  \ar[dl]^{\mbox{\tiny contract $F_{g-1,1,1}$}}  \ar[dr]^{\mbox{\tiny contract $F_{g,1,1}$}} & \\
(\mathbb{P}^1)^n \dblq \SL(2) \ar@{-->}[rr] && Y_{g-1}
}
\end{displaymath}
\end{theorem}
One rich source of flips of GIT quotients is variation of GIT (\cite{DolgachevHu}, \cite{Thaddeus}).  However, comparing the intersection numbers given in Corollary \ref{g-1} with those given in \cite{AlexeevSwinarski} suggests that the flip described above does not arise by varying the linearization on $(\mathbb{P}^1)^n \dblq \SL(2)$.


 \subsection{Finding nef divisors on $\ov{M}_{2(g+1)}$ using the flag morphism}
Given any point $(E;p) \in \ov{{M}}_{1,1}$,  let $f:\ov{M}_{0,n} \longrightarrow \ov{{M}}_{n}$ be the morphism given taking $(C; p_1,\ldots,p_{n}) \in \ov{M}_{0,n}$ to the curve of genus $n$ obtained by attaching $n$ copies of
 $E$ to $C$ by identifying $p$ and $p_i$.  We call this the \emph{flag morphism}.  The results of \cite{GKM} establish a close connection between the properties of a divisor $D$ on $\ov{{M}}_{n}$ and its pullback along this morphism $f^{*}D$.  For instance, every nef divisor on $\ov{M}_{0,n}/S_{n}$ is the pullback of a nef divisor on $\ov{{M}}_{n}$ along this map  (\cite{GKM}*{Thm. $0.7$}); thus, it is natural to ask which nef divisors on $\ov{{M}}_{n}$ pull back to the divisors $\{ \Dl : \ell \in \{1,2,g-1,g\}\}$.  But we can obtain slightly stronger results using \cite{GKM}*{Thm. $0.3$}, which says that a divisor $D$ on $\ov{{M}}_{n}$ which intersects all the $\F$-curves (defined below) nonnegatively is nef if and only if the divisor $f^*D$ is nef on $\ov{M}_{0,n}$.  Thus, we can produce nef divisors on $\ov{{M}}_{n}$ by finding $\F$-divisors on $\ov{{M}}_{n}$ which pull back to $\{ \Dl : \ell \in \{1,2,g-1,g\}\}$.  
We carry this out in Proposition \ref{flagpullback}. 

Throughout this paper we have been using $\F$-curves on $\ov{M}_{0,n}$ (see Section \ref{independentcurves} and the references there for a discussion).  One can also define $\F$-curves on $\ov{M}_{g,n}$.  In \cite{GKM} Theorem 2.2 and Figure 2.3, five types of $\F$-curves on $\ov{M}_{g,n}$ are defined and pictured.  We will refer to these as $\F$-curves of types $(1)-(5)$ in the sequel.  For the reader's convenience, we state a combinatorial definition of an $\F$-divisor $D$ on $\ov{M}_{2(g+1)}$  from \cite{GKM} which we rewrite slightly to fit our situation exactly.  
 \begin{theorem}[\cite{GKM}*{Thm. 2.1}]  \label{F} Let $n=2(g+1)$ and consider the divisor $D=a\lambda-\sum_{i=0}^{g+1}b_i \delta_i$ on 
 $\ov{M}_{2(g+1)}$.  Then $D$ is an $\F$-divisor if and only if it satisfies the following inequalities:
 \begin{enumerate}
 \item $a-12b_0+b_1 \ge 0$;
 \item $b_i \ge 0$;
 \item $2b_0-b_i \ge 0$;
 \item $b_i + b_j \ge b_{i+j}$, for all $i,j \ne 0$;
 \item $b_i+b_j+b_k+b_{\ell} \ge b_{i+j}+b_{i+k}+b_{i+\ell}$, for all $i,j,k,\ell \ne 0$, such that $i+j+k+\ell=2(g+1)$.
 \end{enumerate}
 \end{theorem}
Each of the inequalities $(1)-(5)$ of Theorem \ref{F} is satisfied by a divisor $D$ as long as $D$ nonnegatively intersects 
the corresponding $\F$-curves of types $(1)-(5)$.

In  \cite{GKM}*{Thm. 4.7} a nef divisor which we'll denote $\mathcal{D}$ is defined with the property that $\mathcal{D}$ strictly positively
intersects the $\F$-curve of types $(1)-(4)$ while it intersects the $\F$-curve of types $(5)$ in degree zero.  In particular,
it is shown that $f^*\mathcal{D}$ is trivial.  We will use this divisor $\mathcal{D}$ in Proposition \ref{flagpullback}, and so for
the reader's convenience, we recall its definition.

 \begin{definition}\label{E} On $\ov{M}_{2(g+1)}$ we consider the divisor 
 $$\mathcal{D}=\alpha \lambda - \beta \delta_0-\sum_{i=1}^{g+1} i (2(g+1)-i) \delta_i.$$
 \end{definition}
 \begin{theorem*}[\cite{GKM}*{Thm. 4.7}]  Let $\mathcal{D}$ be the divisor from Definition \ref{E}.  For any choice of $\alpha$ and $\beta$ such that $\alpha > 12 \beta-(2g+1)$, and $2\beta> (g+1)^2$:
 \begin{enumerate}
 \item $\mathcal{D}$ is nef, 
 \item $f^*(\mathcal{D})=0$, and 
 \item $\mathcal{D}$ strictly positively  intersects all the $\operatorname{F}$-curves of type $(1)-(4)$.
\end{enumerate}
 \end{theorem*}
 
 \begin{proposition}\label{flagpullback}For $\ell \in \{1,2,g-1,g\}$,  there is a positive constant $c_{\ell}$ and nonnegative
 constant $d_{\ell}$  such that 
 $$\Dl = f^* (c_{\ell} D_{a,b}^{\ell} +d_{\ell}\mathcal{D}),$$  
 and $c_{\ell}D_{a,b}^{\ell}+d_{\ell} \mathcal{D}$ is a nef divisor on $\ov{{M}}_{2(g+1)}$.  Here:
 \begin{enumerate}
 \item $c_1D_{a,b}^1=a \lambda - b \delta_0-\sum_{i=0}^{\lfloor \frac{g}{2}\rfloor} \delta_{2i+1}$, where $c_1=\frac{1}{4}$,  $b \ge \frac{1}{2}$, and $a \ge 12b-1$;
 \item $c_2D_{a,b}^2=a \lambda - b \delta_0-\sum_{i=0}^{\lfloor \frac{g}{2}\rfloor} \delta_{2i+1}-\sum_{i=1}^{\lfloor \frac{g+1}{2}\rfloor} \frac{4}{3} \ \delta_{2i}$, where $c_2=\frac{4}{3}$, $b \ge \frac{8}{3}$, and $a \ge 12b-1$;
 \item $c_{g-1}D_{a,b}^{g-1}=a \lambda - b \delta_0-\sum_{i=1}^g\frac{i(n-2i+1)}{n-1}\delta_i
 -\frac{3g+2}{n-1}\delta_{g+1}$, \\
 where $c_{g-1}=\frac{1}{(g-1)}$, $b \ge \frac{1}{2}\max \{ \frac{i(n-2i+1)}{n-1} \}_{i=1}^{g}$, and $a \ge 12b-1$;
 \item $c_g D_{a,b}^{g}=a \lambda - b \delta_0-\sum_{i=1}^{g+1}\frac{i(n-2i+1)}{n-1}\delta_i$, where $c_g=1$,  $b \ge \frac{1}{2}\max \{ \frac{i(n-2i+1)}{n-1} \}_{i=1}^{g+1}$ and $a \ge 12b-1$.
 \end{enumerate}   
We may take $d_{1}, d_{2} \geq 0$.  For $\ell \in \{g-1,g\}$, we may choose any $d_{\ell}$ such that $(c_{\ell}D_{a,b}^{\ell}+d_{\ell} \mathcal{D})\cdot \mathcal{C} \ge 0$, where $\mathcal{C}$ is any $\operatorname{F}$ curve of type $(4)$.
 \end{proposition}
 
\begin{proof}
By \cite{GibneyCompositio}*{Lemma 2.4}, the pullback to $\ov{M}_{0,2(g+1)}$ of a divisor $D=a\lambda-\sum_{i=0}^{g+1}b_i \delta_i$ on 
 $\ov{M}_{2(g+1)}$ along the flag map is 
 $$f^*D=\sum_{j=2}^{g+1}(\frac{j(n-j)}{(n-1)}b_1-b_j)B_j, \ \ \mbox{ where } \ \ B_j=\sum_{J \subset \{1,\ldots,n\}, |J|=j}\delta_J,$$
 where $n=2(g+1)$.
 
Using this and the fact that $f^*\mathcal{D} = 0$,  it is straightforward to check that for $\ell \in \{1,2,g-1,g\}$, the divisors $c_{\ell}D_{ab}^{\ell}+d_{\ell} \mathcal{D}$ on $\ov{M}_{0,2(g+1)}$ pull back to $\Dl$.  So it remains to check that for each $\ell \in \{1,2,g-1,g\}$, the divisor $c_{\ell}D_{ab}^{\ell}+d_{\ell}\mathcal{D}$ is nef.  Our main tool for proving that divisors on $\ov{M}_{0,2(g+1)}$ are nef will be to check the conditions of Theorem \ref{F} and apply \cite{GKM}*{Thm. $0.3$}.
 



First, we will analyze the cases $\ell=1$ and $\ell =2$.

It is easy to check that conditions $(1)-(3)$ of Theorem \ref{F} hold for $D^{\ell}_{ab}$ for all $\ell \in \{1,2\}$ since we chose $a$ and $b$ to make this happen.

 Condition $(5)$ is just the combinatorial formulation that 
 $(c_{\ell}D_{ab}^{\ell}) \cdot F^{2(g+1)}_{i,j,k,\ell} \ge 0$,
where  $F^{2(g+1)}_{i,j,k,\ell}$ on $\overline{M}_n$ is the image of the $\F$-curve $F_{i,j,k,\ell}$ on $\overline{M}_{0,n}$
  under the flag map.  In other words, this is equivalent to 
 $$f^*(c_{\ell} D_{ab}^{\ell}) \cdot F_{i,j,k,\ell}=\Dl  \cdot F_{i,j,k,\ell} \ge 0,$$
 which holds since $\Dl$ is nef.

This leaves condition $(4)$.

We check  condition $(4)$ for $D^{1}_{ab}$.  Since $b_{k}$ only depends on the parity of $k$, we need only consider two cases.  If $i$ and $j$ have the same parity, then the equality reads $2 \geq 0$.  If $i$ and $j$ have opposite parity, then the inequality reads $1 \geq 1$.  So we may conclude that $D^{1}_{a,b}$ is nef. By \cite{GKM}*{Thm. 4.7} $\mathcal{D}$ is also nef,  and hence for any nonnegative $c_{1}, d_1$, the divisor $c_{1} D^{1}_{a,b}+d_{1}\mathcal{D}$ is nef.  


 We next check condition $(4)$ for $D^{2}_{ab}$.   Note that $b_i+b_j \geq 2$ while $b_{i+j} \leq \frac{4}{3}$, so $(4)$ holds.  Thus $D^{2}_{ab}$ is nef, and hence for any nonnegative $c_{2}, d_2$, the divisor $c_{2} D^{2}_{a,b}+d_{2}\mathcal{D}$ is nef.

 Next, we will analyze the cases $\ell=g-1$ and $\ell =g$.  The two divisors $D^{g-1}_{ab}$ and $D^{g}_{ab}$ are not nef by themselves, as condition $(4)$ does not always hold.  In particular, it is necessary to choose a sufficiently large $d_{g-1}$ and $d_g$.

As before, the hypotheses on $a$, $b$, $\alpha$, and $\beta$ ensure that conditions $(1)-(3)$ are satisfied .  Condition $(5)$ follows just as it did above since we know that $c_{\ell} D^{\ell}_{a,b}+d_{\ell}\mathcal{D}$ pulls back to a nef divisor on $\ov{M}_{0,n}$.  This leaves condition $(4)$.  We know that $\mathcal{D}$ has positive intersection with $\F$-curves of type $(4)$.  Therefore, we simply need to choose $d_{\ell}$ sufficiently large that $(c_{\ell}D_{a,b}^{\ell}+d_{\ell} \mathcal{D})\cdot \mathcal{C} \ge 0$ for all $\F$-curves of type $(4)$.  There are only finitely many such curves to check (or, only finitely many inequalities of type $b_{i} + b_{j} \ge b_{i+j}$), so this can be arranged.  Then $c_{\ell} D^{\ell}_{a,b}+d_{\ell}\mathcal{D}$ is $\F$-nef, and hence nef by \cite{GKM}*{Thm. $0.3$}.


\end{proof}

\begin{remark}
 If given $\ell \in \{1,2,g-1,g\}$, one could find $a$, $b$ and $d_{\ell}$ such that 
$c_{\ell} D_{a,b}^{\ell} +d_{\ell}\mathcal{D}$ is semiample, then it would follow from Proposition \ref{flagpullback} that the morphism given by $\Dl$ would factor through the flag locus.  At the time of this writing, there are no divisors on $\overline{M}_{2(g+1)}$ that are known to be nef but not semiample.
\end{remark}

\end{document}